\documentclass[a4paper, 12pt]{amsart}
\usepackage{hyperref, amssymb, tikz}

\newtheorem{thm}{Theorem}[section]
\newtheorem{prp}[thm]{Proposition}
\newtheorem{lem}[thm]{Lemma}
\newtheorem{cor}[thm]{Corollary}
\theoremstyle{definition}
\newtheorem{dfn}[thm]{Definition}
\newtheorem{ntn}[thm]{Notation}
\theoremstyle{remark}
\newtheorem{rmk}[thm]{Remark}

\numberwithin{equation}{section}

\newcommand{\Bb}{\mathcal{B}}

\newcommand{\Mm}{\mathcal{M}}
\newcommand{\Pp}{\mathcal{P}}
\newcommand{\Oo}{\mathcal{O}}
\newcommand{\Tt}{\mathcal{T}}

\newcommand{\CC}{\mathbb{C}}
\newcommand{\NN}{\mathbb{N}}

\newcommand{\RR}{\mathbb{R}}
\newcommand{\TT}{\mathbb{T}}
\newcommand{\ZZ}{\mathbb{Z}}

\newcommand{\Aut}{\operatorname{Aut}}
\renewcommand{\emptyset}{\varnothing}

\newcommand{\lsp}{\operatorname{span}}
\newcommand{\clsp}{\operatorname{\overline{\lsp}}}

\title[KMS states on generalised Bunce--Deddens algebras]{KMS states on generalised Bunce--Deddens algebras and their Toeplitz extensions}
\author{David Robertson}
\email[D. Robertson]{dave84robertson@gmail.com}
\author{James Rout}
\email[J. Rout]{jdr749@uowmail.edu.au}
\author{Aidan Sims}
\email[A. Sims]{asims@uow.edu.au}
\address{School of Mathematics and Applied Statistics\\
University of Wollongong\\
Wollongong NSW 2522\\ AUSTRALIA}

\date{\today}
\thanks{This research was supported by the Australian Research Council.}

\begin{document}

\begin{abstract}
We study the generalised Bunce--Deddens algebras and their Toeplitz extensions
constructed by Kribs and Solel from a directed graph and a sequence $\omega$ of positive
integers. We describe both of these $C^*$-algebras in terms of novel universal properties, and
prove uniqueness theorems for them; if $\omega$ determines an infinite supernatural
number, then no aperiodicity hypothesis is needed in our uniqueness theorem for the
generalised Bunce--Deddens algebra. We calculate the KMS states for the gauge action in
the Toeplitz algebra when the underlying graph is finite. We deduce that the generalised
Bunce--Deddens algebra is simple if and only if it supports exactly one KMS state, and
this is equivalent to the terms in the sequence $\omega$ all being coprime with the
period of the underlying graph.
\end{abstract}

\maketitle

\section{Introduction}

Every Cuntz--Krieger algebra $\Oo_A$ carries a gauge action of $\TT$ which lifts to an
action $\alpha$ of $\RR$. Enomoto, Fujii and Watatani \cite{EnomotoFujiiEtAl:MJ84} proved
that when $A$ is irreducible, $(\Oo_A, \alpha)$ has a unique KMS state, which occurs at
inverse temperature equal to the logarithm $\ln\rho(A)$ of the spectral radius of $A$.
Exel and Laca \cite{ExelLaca:CMP03} extended this result to Cuntz--Krieger algebras of
infinite matrices and also described the KMS states of their Toeplitz extensions. More
recently, an Huef, Laca, Raeburn and Sims extended these results to the graph algebras of
finite graphs \cite{anHuefLacaEtAl:JMAA2013} and $C^*$-algebras associated to higher-rank
graphs \cite{anHuefLacaEtAl:JFA14}. In each case, the Toeplitz extension has many more
KMS states than the Cuntz--Krieger algebra, and encodes more information about the
underlying object.

In \cite{KribsSolel:JAMS07}, Kribs and Solel studied $C^*$-algebras generated by periodic
weighted-shift operators on the path spaces of directed graphs. They showed that the
$C^*$-algebra generated by all such operators can be realised as a direct limit of graph
algebras. Specifically, given $n > 0$, they construct a graph $E(n)$ with vertex set
$E^{<n}$, the space of paths in $E$ of length at most $n-1$, and they exhibited
inclusions $\Tt C^*(E(n)) \hookrightarrow \Tt C^*(E(mn))$. Upon restriction to the
canonical abelian subalgebra in $\Tt C^*(E(mn))$, these inclusions are compatible with a
natural surjection $E^{<mn} \to E^{<n}$, so $\varinjlim \Tt C^*(E(n))$ has an abelian
subalgebra isomorphic to $C_0(\varprojlim E^{<n})$. This construction has recently been
used to calculate the nuclear dimension of graph algebras and Kirchberg algebras
\cite{RuizSimsEtAl:xx14, RuizSimsEtAl:AMxx}.

Here we study the structure of the Kribs--Solel algebras and their Toeplitz extensions,
and calculate the KMS states of the associated dynamics. We start in
Section~\ref{sec:universal} by giving a universal description of the Kribs--Solel algebra
$C^*(E, \omega)$ of a directed graph $E$ corresponding to a sequence $\omega =
(n_k)^\infty_{k=1}$ of positive integers as generated by a copy of $C^*(E)$ and a copy of
$C_0(\varprojlim E^{< n_k})$. We give an analogous description of the Toeplitz extension
$\Tt(E, \omega)$. Our approach clarifies the structure of these algebras, and in
particular makes transparent the fact that $C^*(E, \omega)$ and $\Tt(E,\omega)$ depend
only on $E$ and the supernatural number determined by $\omega$ (see
Proposition~\ref{prp:equiv omegas}).

Kribs and Solel use a topological graph $E(\infty)$ in the sense of Katsura to study some
properties of their direct-limit algebras. They show that $C^*(E,
\omega)$ is isomorphic to the topological-graph $C^*$-algebra $C^*(E(\infty))$, allowing
them to plug into Katsura's powerful structure theory. In Section~\ref{sec:topgraph} we
provide a slightly different description of $E(\infty)$ that we feel clarifies the
construction somewhat, and study its structure in greater depth than appears in
\cite{KribsSolel:JAMS07}. In particular, when $E$ is finite and strongly connected, we
show how to decompose $E(\infty)$ into irreducible components using Perron-Frobenius
theory for the matrices $A_E^{n_k}$.

In Section~\ref{sec:uniqueness} we prove uniqueness theorems for $C^*(E, \omega)$ and
$\Tt(E, \omega)$. The uniqueness theorem for $\Tt(E, \omega)$
(Proposition~\ref{prp:coburn}) is analogous to that for the Toeplitz extension of a graph
algebra, and we prove it using that technology. Interestingly, our Cuntz--Krieger
uniqueness theorem (Theorem~\ref{thm:CKUT}) for $C^*(E,\omega)$ requires no aperiodicity
hypothesis, emphasising Kribs and Solel's view of these algebras as generalised
Bunce--Deddens algebras. We obtain this result by combining Katsura's uniqueness theorem
for topological graph $C^*$-algebras with Kribs and Solel's observation that their
topological graph $E(\infty)$ has no loops. This also leads to a very satisfactory
characterisation of ideal-structure for $C^*(E, \omega)$ for finite, strongly connected $E$:
$C^*(E, \omega)$ decomposes as a direct sum of simple subalgebras indexed by the finite
group of integers modulo the greatest common divisor of the supernatural number $\omega$
and the period of the graph $E$ in the sense of Perron--Frobenius theory. In particular,
$C^*(E,\omega)$ is simple if and only if $\omega$ is coprime to the period of $E$.

In Section~\ref{sec:KMS}, we focus on finite strongly connected graphs $E$, and study the
KMS states for the gauge-dynamics on $\Tt(E, \omega)$, paying attention to those which
factor through $C^*(E,\omega)$. Our analysis follows the broad lines of
\cite{ExelLaca:CMP03, LacaRaeburn:AM10}, but the inverse-limit structure of the spectrum
of the diagonal in $\Tt(E, \omega)$ introduces some interesting wrinkles. We reinterpret
the KMS condition for states on $\Tt(E,\omega)$ as a subinvariance condition for an
operator on the space of signed Borel measures on $\varprojlim E^{< n_k}$
(Theorem~\ref{thm:KMSchar}). To construct KMS states on the Toeplitz algebra of a graph
$E$, one makes use of the path-space representation on $\ell^2(E^*)$. It is not \emph{a
priori} clear how to construct a corresponding representation of $\Tt(E, \omega)$ from
the Kribs--Solel approach, but our universal description of $\Tt(E, \omega)$ suggests a
solution. We use this representation to construct KMS$_\beta$ states for all $\beta >
\ln\rho(A_E)$ (Proposition~\ref{prp:constructKMS}), and show that there is an affine
isomorphism between the KMS$_\beta$ simplex of $\Tt(E,\omega)$ and the simplex of
probability measures on $\varprojlim E^{<n_k}$ (Corollary~\ref{cor:normalisedaffine}).

Finally, we investigate which KMS states factor through $C^*(E,\omega)$. In contrast with
\cite{EnomotoFujiiEtAl:MJ84, anHuefLacaEtAl:JMAA2013}, strong connectedness of $E$ is not
sufficient to ensure that $C^*(E, \omega)$ admits a unique KMS state; rather, the
extremal KMS states of $C^*(E, \omega)$ correspond precisely to the direct summands
described in Section~\ref{sec:uniqueness}. Following the approach of
\cite{anHuefLacaEtAl:JFA15}, we describe a formula which always determines a
KMS$_{\ln\rho(A_E)}$ state $\phi$ of $C^*(E,\omega)$. Restricting this state to each
direct summand of $C^*(E, \omega)$ and normalising yields a family of KMS states all at
the same inverse temperature, and we use the results of \cite{EnomotoFujiiEtAl:MJ84,
anHuefLacaEtAl:ETDS14} to show that there cannot be any KMS states for $C^*(E, \omega)$
at any other temperatures. We prove that these states are precisely the extremal KMS
states of $C^*(E, \omega)$. We deduce that $\phi$ is the only KMS state of
$C^*(E,\omega)$ if and only if $\omega$ is coprime with the period of $E$, and hence if
and only if $C^*(E)$ is simple; we further show that this is equivalent to $\phi$ being a
factor state.

\section{Background}\label{sec:background}

\subsection{Directed graphs and their \texorpdfstring{$C^*$}{C*}-algebras}
We use the convention for graph $C^*$-algebras appearing in Raeburn's book
\cite{Raeburn:Graphalgebras05}. So if $E = (E^0, E^1, r, s)$ is a directed graph, then a
path in $E$ is a word $\mu = e_1\dots e_n$ in $E^1$ such that $s(e_i) = r(e_{i+1})$ for
all $i$, and we write $r(\mu) = r(e_1)$, $s(\mu) = s(e_n)$, and $|\mu| = n$. As usual, we
denote by $E^*$ the collection of paths of finite length, and $E^n := \{\mu \in E^* : |\mu| = n\}$;
we also write $E^{<n} := \{\mu \in E^* : |\mu| < n\}$. We borrow the convention from
the higher-rank graph literature in which we write, for example $vE^*$ for
$\{\mu \in E^* : r(\mu) = v\}$, and $v E^1 w$ for $\{e \in E^1 : r(e) = v\text{ and } s(e)= w\}$.
The adjacency matrix of $E$ is then the $E^0 \times E^0$ integer matrix with $A_E(v,w) = |vE^1w|$.

We say that $E$ is \emph{row-finite} if $vE^1$ is finite for all $v \in E^0$, and that
$E$ has no sources if each $vE^1$ is nonempty.

If $E$ is row-finite and has no sources, then a Toeplitz--Cuntz--Krieger $E$-family in a
$C^*$-algebra $A$ is a pair $(s,p)$, where $s = \{s_e : e \in E^1\} \subseteq A$ is a
collection of partial isometries and $\{p_v : v \in E^0\} \subseteq A$ is a set of
mutually orthogonal projections such that $s^*_e s_e = p_{s(e)}$ for all $e \in E^1$, and
$p_v \ge \sum_{e \in vE^1} s_e s^*_e$ for all $v \in E^0$. If equality holds in the
second relation (for all $v$), then $(s,p)$ is a Cuntz--Krieger $E$-family.

The Toeplitz algebra $\Tt C^*(E)$ is the universal $C^*$-algebra generated by a
Toeplitz--Cuntz--Krieger family (\cite{FowlerRaeburn:IUMJ99}) and the graph algebra
$C^*(E)$ is the universal $C^*$-algebra generated by a Cuntz--Krieger $E$-family
\cite[Proposition~1.21]{Raeburn:Graphalgebras05}.

Kribs and Solel describe their generalised Bunce--Deddens algebras as direct limits of
graph $C^*$-algebras obtained from the following construction
\cite[Theorem~4.2]{KribsSolel:JAMS07}.

Let $E = (E^0, E^1, r, s)$ be a row-finite directed graph with no sources, and fix $n \in
\NN\setminus\{0\}$. Define
\begin{gather*}
E(n)^0 := E^{<n},\qquad E(n)^1 := \{(e,\mu) : e \in E^1, \mu \in s(e)E^{<n}\}\\
s_n(e,\mu) := \mu\qquad\text{ and }\qquad
r_n(e,\mu) = \begin{cases}
        e\mu &\text{ if $|\mu| < n-1$} \\
        r(e) &\text{ if $|\mu| = n-1$.}
    \end{cases}
\end{gather*}
Then $E(n) = (E(n)^0, E(n)^1, r_n, s_n)$ is a row-finite directed graph with no sources.
For $\mu \in E^*$, we write
$[\mu]_n$ for the unique element of $E^{<n}$ such that $\mu = [\mu]_n\mu'$ for some
$\mu'$ with $|\mu'| \in n\NN$; we think of $[\mu]_n$ as the residue of $\mu$ modulo $n$.

It is easy to check that there is a bijection from $\{(\mu,\nu) : \mu \in E^*, \nu \in
s(\mu)E^{<n}\}$ to $E(n)^*$ that carries $(\mu,\nu)$ to $(\mu_1, [\mu_2\dots\mu_{|\mu|}\nu]_n)(\mu_2,
[\mu_3\dots \mu_{|\mu|}\nu]_n) \dots (\mu_{|\mu|}, \nu)$. We frequently use this
bijection to identify $E(n)^*$ with $\{(\mu,\nu) : \mu \in E^*, \nu \in s(\mu)E^{<n}\}$,
and we then have $s_n(\mu,\nu) = \nu$, and $r_n(\mu,\nu) = [\mu\nu]_n$. This implies, in
particular, that the lengths of the paths $r_n(\mu,\nu)$ and $s_n(\mu,\nu)$ in $E^{<n}$
differ by $|\mu|$ modulo $n$. Thus, for $v,w \in E^0 \subseteq E^{<n}$, we have
\begin{equation}\label{eq:connectivity in E(n)}
v E(n)^* w \not= \emptyset \quad\text{ if and only if }\quad
    v E^{jn} w \not= \emptyset\text{ for some $j \in \NN$.}
\end{equation}

\subsection{The KMS condition}
We use the definition of KMS states given in
\cite[Definition~5.3.1]{BratteliRobinson:OAQSMvII}. Let $(A, \RR, \alpha)$ be a
$C^*$-dynamical system. An element $a \in A$ is analytic for $\alpha$ if $t \mapsto
\alpha_t(a)$ extends to an entire function $z \mapsto \alpha_z(a)$ on $\CC$.
Let $A_\alpha$ denote the collection
of analytic elements of $A$. A state $\phi$ of $A$ is said to be a KMS state at inverse
temperature $\beta \in \RR \setminus\{0\}$ if
\[
\phi(ab) = \phi(b\alpha_{i\beta}(a))\quad\text{ for all $a,b \in A_\alpha$.}
\]
It suffices to verify this KMS condition on any $\alpha$-invariant set of analytic elements
spanning a dense subspace of $A$. Proposition~5.3.3 of \cite{BratteliRobinson:OAQSMvII}
says that if $\phi$ is KMS$_\beta$ for $\alpha$, then $\phi$ is $\alpha$-invariant. If
$\beta = 0$, then the KMS condition above reduces to requiring that $\phi$ is a trace,
and we then impose $\alpha$-invariance as an additional requirement.

\subsection{The Perron--Frobenius theorem}
Let $X$ be a finite set. A matrix $A \in M_X(\CC)$ is \emph{irreducible} if, for all $x,y
\in X$, there exists $n \in \NN$ such that $A^n(x,y) \not= 0$. We say that a matrix is
nonnegative if all of its entries are nonnegative.

Let $A$ be an irreducible nonnegative matrix. The Perron--Frobenius theorem (see, for
example, \cite[Theorem~1.5]{Seneta:NNMMC}) says that the spectral radius $\rho(A)$ is an
eigenvalue of $A$ with a positive eigenvector, and that $\rho(A)$ is a simple root of the
characteristic polynomial of $A$. We call the unique positive eigenvector with eigenvalue
$\rho(A)$ and unit $1$-norm the \emph{unimodular Perron--Frobenius eigenvector} of $A$.

\subsection{The space of finite signed Borel measures}
If $M$ is a $\sigma$-algebra of subsets of a set $X$, then a real-valued function $m$
defined on $M$ is said to be a finite signed measure if $m(\emptyset) = 0$ and $m$ is
completely additive.

Suppose that $X$ is a compact Hausdorff space. We denote by $\Mm(X)$ the space of all
finite signed Borel measures on $X$, by $\Mm^+(X)$ the subset of $\Mm(X)$ consisting of
positive Borel measures, and by $\Mm^+_1(X)$ the subset of $\Mm^+(X)$ consisting of
probability measures on $X$.

Let $m \in \Mm(X)$. By the Hahn decomposition theorem
\cite[Theorem~8.2]{Bartle:ElementsIntegration} there are sets $P,N \subseteq X$ such that
$X = P \cup N$ and $P \cap N = \emptyset$, and such that $m(E \cap P) \ge 0$ and $m(E
\cap N) < 0$ for all Borel $E \subseteq X$.

Let $m^+$ and $m^-$ be given by $m^+(E) = m(E \cap P)$ and $m^-(E) = -m(E \cap N)$ for
Borel $E$. Then $m^+, m^- \in \Mm^+(X)$. The Jordan decomposition theorem
\cite[Theorem~8.5]{Bartle:ElementsIntegration} says that $m = m^+ - m^-$ and that if $m',
m'' \in \Mm^+(X)$ satisfy $m = m' - m''$, then $m'(E) \ge m^+(E)$ and $m''(E) \ge m^-(E)$
for all Borel $E \subseteq X$.

The space $\Mm(X)$ of finite signed measures is a real Banach space under the norm
$\|m\| = m^+(X) + m^-(X)$.

\section{The Kribs--Solel algebras and their Toeplitz extensions}\label{sec:universal}

In this section, we describe an alternative presentation of Kribs and Solel's
$C^*$-algebras $\Tt C^*(E(n))$ and $C^*(E(n))$, and of their direct-limit algebras $A(n)$
and $B(n)$. We show that $\Tt C^*(E(n))$ is the universal $C^*$-algebra generated by a
Toeplitz--Cuntz--Krieger $E$-family and mutually orthogonal projections indexed by $E^{<n}$.
This presentation has the advantage that the connecting maps $\Tt C^*(E(n))
\hookrightarrow \Tt C^*(E(nm))$ have a particularly simple form: they preserve the
generating Toeplitz--Cuntz-Krieger family, and resolve the projection associated to each
$\mu \in E^{<n}$ into a sum of projections associated to paths of the form $\mu\tau \in
E^{<nm}$. This leads to a very natural presentation of $A(n)$ in terms of a
Toeplitz--Cuntz--Krieger $E$-family and a representation of the algebra of continuous
functions on a natural projective limit of the $E^{<n}$. We show that all of this
descends naturally to the $C^*(E(n))$ and $B(n)$.

\begin{dfn}\label{dfn:Toeplitz nrep}
Let $E$ be a row-finite directed graph with no sources, and fix $n \in
\NN\setminus\{0\}$. A \emph{Toeplitz $n$-representation} of $E$ in a $C^*$-algebra $A$ is a
triple $(T, Q, \Theta)$ where $(T,Q)$ is a Toeplitz--Cuntz--Krieger $E$-family in $A$,
and $\Theta = \{\Theta_\mu : \mu \in E^{<n}\}$ is a collection of mutually orthogonal
projections such that $Q_v = \sum_{\mu \in vE^{<n}} \Theta_\mu$ for all $v \in E^0$, and
\begin{equation}\label{eq:nrep rel}
T^*_e \Theta_\mu =
    \begin{cases}
        \Theta_{\mu'} T^*_e &\text{ if $\mu = e\mu'$}\\
        \sum_{e\nu \in E^n} \Theta_\nu T^*_e &\text{ if $\mu = r(e)$} \\
        0 &\text{ otherwise.}
    \end{cases}
\end{equation}
If $(T,Q)$ is a Cuntz--Krieger $E$-family, we call $(T, Q, \Theta)$ a
\emph{Cuntz--Krieger $n$-representation} of $E$.
\end{dfn}

We show that Kribs and Solel's $\Tt C^*(E(n))$ is universal for Toeplitz
$n$-representations of $E$ and that $C^*(E(n))$ is universal for Cuntz--Krieger
$n$-representations. We first describe a convenient family of spanning elements. We will
need the following notation: given a directed graph $E$, $n > 0$ and $\mu \in E^*$, we
write $\tau_n(\mu)$ for the unique element of $E^{<n}$ such that $\mu = \mu' \tau_n(\mu)$
with $|\mu'| \in n\NN$; so $|\tau_n(\mu)| \equiv |\mu| \;(\operatorname{mod}\,n)$, and $\mu
= \mu'\tau_n(\mu)$.

\begin{lem}\label{lem:"covariance"}
Let $E$ be a row-finite directed graph with no sources, and take $n > 0$. Let $(T, Q,
\Theta)$ be a Toeplitz $n$-representation of $E$, and fix $\mu \in E^*$ and $\alpha \in
E^{<n}$.
\begin{enumerate}
\item\label{it:commute mu 1} If $|\mu| \in n\NN$, then $T^*_\mu \Theta_{r(\mu)} =
    \Theta_{s(\mu)} T^*_\mu$.
\item\label{it:commute mu 2}
\[
    T^*_\mu \Theta_\alpha = \begin{cases}
        \Theta_{\alpha'}T^*_\mu &\text{ if $\alpha = \mu\alpha'$}\\
        \Theta_{s(\mu)} T^*_\mu &\text{ if $\mu = \alpha\mu'$ and $|\mu'| \in n\NN$}\\
        \sum_{|\tau_n(\mu')\lambda| = n} \Theta_\lambda T^*_\mu &\text{ if $\mu = \alpha\mu'$ and $|\mu'| \not\in n\NN$}\\
        0 &\text{ otherwise.}
    \end{cases}
\]
\end{enumerate}
\end{lem}
\begin{proof}
(\ref{it:commute mu 1}) We calculate
\begin{align}
T^*_{\mu} \Theta_{r(\mu)}
    &= T^*_{\mu_{|\mu|}} \cdots T^*_{\mu_2} T^*_{\mu_1} \Theta_{r(\mu_1)}
    = T^*_{\mu_{|\mu|}} \cdots T^*_{\mu_2} \Big(\sum_{\mu_1\lambda \in E^n} \Theta_{\lambda} T^*_{\mu_1}\Big) \nonumber\\
    &= T^*_{\mu_{|\mu|}} \cdots T^*_{\mu_3} \Big(\sum_{\mu_1\mu_2\lambda \in E^n} \Theta_{\lambda} T^*_{\mu_1\mu_2}\Big)
    = \cdots = \Theta_{s(\mu)} T^*_\mu.\label{eq:goodcalc}
\end{align}

(\ref{it:commute mu 2}) First suppose that $\alpha = \mu\alpha'$. Then
\[
T^*_\mu \Theta_\alpha
    = T^*_{\mu_{|\mu|}} \cdots T^*_{\mu_1} \Theta_\alpha
    = T^*_{\mu_{|\mu|}} \cdots T^*_{\mu_2} \Theta_{\alpha_2\cdots \alpha_n} T^*_{\mu_1}
    = \cdots = \Theta_{\alpha'} T^*_\mu.
\]
Now suppose that $\mu = \alpha\mu'$. Write $\mu' = \mu''\tau_n(\mu')$. Then $|\mu''| \in
n\NN$, so we calculate, using part~(\ref{it:commute mu 1}) at the fourth equality,
\[
T^*_\mu\Theta_\alpha
    = T^*_{\mu'} T^*_\alpha \Theta_\alpha
    = T^*_{\mu'} \Theta_{s(\alpha)} T^*_\alpha
    = T_{\tau_n(\mu')}^* T_{\mu''}^* \Theta_{r(\mu'')} T^*_\alpha
    = T_{\tau_n(\mu')}^* \Theta_{s(\mu'')} T^*_{\alpha\mu''}.
\]
If $|\mu'| \in n\NN$, then $\alpha\mu'' = \mu$ and $\tau_n(\mu') = s(\mu)$, so the
preceding displayed equation gives $T^*_\mu \Theta_{\alpha} = \Theta_{s(\mu)} T^*_\mu$.
Otherwise, we repeat the first $|\mu''|$ steps of the calculation~\eqref{eq:goodcalc} to
obtain
\[
T^*_\mu\Theta_\alpha
    = \sum_{|\tau_n(\mu')\lambda| = n} \Theta_\lambda T^*_\mu.
\]
Finally, if $\mu\not= \alpha\mu'$ and $\alpha\not= \mu\alpha'$, then we can write $\mu =
\lambda e\mu'$ and $\alpha = \lambda f \alpha'$ for distinct $e,f \in E^1$. Using the
first case in part~(\ref{it:commute mu 2}), we obtain
\[
T^*_\mu \Theta_\alpha
    = T^*_{\mu'} T^*_e \Theta_{f\alpha'} T^*_\lambda,
\]
which is zero by the displayed relation in Definition~\ref{dfn:Toeplitz nrep}.
\end{proof}

\begin{lem}\label{lem:mpctn formula}
Let $E$ be a row-finite directed graph with no sources, take $n > 0$ and suppose that
$(T, Q, \Theta)$ is a Toeplitz $n$-representation of $E$. For $\alpha,\beta, \gamma,
\delta \in E^*$ and $\mu,\nu \in E^{<n}$,
\[
(T_\alpha \Theta_\mu T^*_\beta) (T_\gamma \Theta_\nu T^*_\delta)
    = \begin{cases}
        T_\alpha \Theta_\mu T^*_{\delta\beta'}
            &\text{ if $\beta = \gamma\beta'$ and $\nu = \beta'\mu$} \\
        T_\alpha \Theta_\mu T^*_{\delta\nu\rho}
            &\text{ if $\beta = \gamma\nu\rho$ with $|\rho\mu| \in n\NN$}\\
        T_{\alpha\gamma'} \Theta_\nu T^*_\delta
            &\text{ if $\gamma = \beta\gamma'$ and $\mu = \gamma'\nu$} \\
        T_{\alpha\mu\rho} \Theta_\nu T^*_\delta
            &\text{ if $\gamma = \beta\mu\rho$ with $|\rho\nu| \in n\NN$}\\
        0   &\text{ otherwise.}
    \end{cases}
\]
\end{lem}
\begin{proof}
We consider the case where $|\beta| \ge |\gamma|$; the case where $|\gamma|
> |\beta|$ will then follow by taking adjoints. By
\cite[Corollary~1.14(b)]{Raeburn:Graphalgebras05}, we have
\[
(T_\alpha \Theta_\mu T^*_\beta) (T_\gamma \Theta_\nu T^*_\delta)
    = \begin{cases}
        T_\alpha \Theta_\mu T^*_{\beta'} \Theta_\nu T^*_\delta
            &\text{ if $\beta = \gamma\beta'$} \\
        0   &\text{ otherwise.}
    \end{cases}
\]
Suppose that $\beta = \gamma\beta'$. By Lemma~\ref{lem:"covariance"}(\ref{it:commute mu
2}) we have
\begin{align*}
(T_\alpha \Theta_\mu T^*_\beta) (T_\gamma \Theta_\nu T^*_\delta)
    &= \begin{cases}
        T_\alpha \Theta_\mu \Theta_{\nu'} T^*_{\delta\beta'}
            &\text{ if $\nu = \beta'\nu'$} \\
        T_\alpha \Theta_\mu \Theta_{s(\beta')} T^*_{\delta\beta'}
            &\text{ if $\beta' = \nu\rho$ with $|\rho| \in n\NN$} \\
        T_\alpha \Theta_\mu \sum_{\tau_n(\rho)\lambda \in E^n} \Theta_\lambda T^*_{\delta\beta'}
            &\text{ if $\beta' = \nu\rho$ with $|\rho| \not\in n\NN$} \\
        0   &\text{ otherwise.}
    \end{cases}\\
    &= \begin{cases}
        T_\alpha \Theta_\mu T^*_{\delta\beta'}
            &\text{ if $\nu = \beta'\mu$} \\
            &\text{ \quad or $\beta' = \nu\rho$ with $|\rho| \in n\NN$ and $\mu = s(\beta)$}\\
            &\text{ \quad or $\beta' = \nu\rho$ and $\tau_n(\rho)\mu \in E^n$,}\\
        0 &\text{ otherwise.}
    \end{cases}
\end{align*}

Since $\tau_n(\rho)\mu \in E^n$ if and only if $|\rho\mu| \in n\NN$, the result follows.
\end{proof}

\begin{thm}\label{thm:TEn and CEn}
Let $E$ be a row-finite directed graph with no sources, and let $(t_{(e,\mu)}, q_\mu)$
be the universal Toeplitz--Cuntz--Krieger $E(n)$-family in $\Tt C^*(E(n))$. Then the
elements
\[
t_{n,e} := \sum_{(e,\mu) \in E(n)^1} t_{(e,\mu)}, \quad
    q_{n,v} := \sum_{\mu \in vE^{<n}} q_\mu,\quad\text{ and }\quad
    \theta_{n,\mu} := q_\mu
\]
constitute a Toeplitz $n$-representation of $E$ and generate $\Tt C^*(E(n))$. For every
Toeplitz $n$-representation $(T, Q, \Theta)$ of $E$ in a $C^*$-algebra $B$, there is a
$C^*$-homomorphism $\pi_{T, Q, \Theta} : \Tt C^*(E(n)) \to B$ such that
$\pi_{T,Q,\Theta}(t_{n,e}) = T_e$, $\pi_{T,Q,\Theta}(q_{n,v}) = Q_v$ and
$\pi_{T,Q,\Theta}(\theta_{n,\mu}) = \Theta_\mu$.

If $(T, Q)$ is a Cuntz--Krieger $E$-family, then $\pi_{T, Q, \Theta}$ factors through a
homomorphism $\tilde\pi_{T, Q, \Theta} : C^*(E(n)) \to B$.
\end{thm}
\begin{proof}
Routine calculations show that $(t, q, \theta)$ is a Toeplitz $n$-representation of $E$.
We have $t_{(e,\mu)} = t_{n,e} \theta_{n,\mu}$ for each $(e,\mu) \in E(n)^1$ and $q_\mu =
\theta_{n,\mu}$ for each $\mu \in E(n)^0$, and so the $t_{n,e}$, the $q_{n,v}$ and the
$\theta_{n,\mu}$ generate $\Tt C^*(E(n))$.

Fix a Toeplitz $n$-representation $(T, Q, \Theta)$. Routine calculations show that the
elements $T_{(e,\mu)} := T_e \Theta_\mu$ and $Q_\mu := \Theta_\mu$ form a
Toeplitz--Cuntz--Krieger $E(n)$ family, and so induce the desired homomorphism $\pi_{T,
Q, \Theta}$. For each $v \in E^0$, we have
\[
\sum_{\mu \in vE^{<n}} \sum_{(e,\nu) \in \mu E(n)^1} T_{(e,\nu)} T^*_{(e,\nu)}
    = \sum_{e \in v E^1} \sum_{(e,\nu) \in E(n)^1} T_{(e,\nu)} T^*_{(e,\nu)}
    = \sum_{e \in vE^1} T_e T^*_e.
\]
So if $(T, Q)$ is a Cuntz--Krieger $E$-family, then the $T_{(e,\mu)}$ and $Q_\mu$ form a
Cuntz--Krieger $E(n)$ family and so $\pi_{T, Q, \Theta}$ factors through $\tilde\pi_{T,
Q, \Theta} : C^*(E(n)) \to B$.
\end{proof}

\begin{ntn}
Using Theorem~\ref{thm:TEn and CEn}, we write $\Tt(E, n)$ for $\Tt C^*(E(n))$ and regard
it as the universal $C^*$-algebra generated by a Toeplitz $n$-representation
$(t_{n,e}, q_{n,v}, \theta_{n,\mu})$ of $E$. We also write $C^*(E, n)$ for
$C^*(E(n))$, and regard it as the universal $C^*$-algebra generated by a Cuntz--Krieger
$n$-representation $(s_{n,e}, p_{n,v}, \varepsilon_{n,\mu})$.
\end{ntn}

Next, we describe the homomorphisms $\Tt C^*(E(n)) \hookrightarrow \Tt C^*(E(mn))$ and
$C^*(E(n)) \hookrightarrow C^*(E(mn))$ of Kribs and Solel in terms of the universal
properties just described.

\begin{prp}\label{prp:inclusions}
Let $E$ be a row-finite directed graph with no sources. Take integers $m,n \ge 1$. There
is an injective homomorphism $i_{n,mn} : \Tt(E, n) \to \Tt(E, mn)$ such that
\[
i_{n,mn}(t_{n,e}) = t_{mn,e},\qquad
i_{n,mn}(q_{n,v}) = q_{mn,v},\quad\text{ and }\quad
i_{n,mn}(\theta_{n,\mu}) = \sum_{\nu \in E^{<mn},\, [\nu]_n = \mu} \theta_{mn,\nu}.
\]
Moreover $i_{n,mn}$ descends to an injection of $C^*(E,n)$ into $C^*(E, mn)$.
\end{prp}
\begin{proof}
For $e \in E^1$, $v \in E^0$ and $\mu \in E^{<n}$, define $T_e := t_{mn,e}$, $Q_v :=
q_{mn,v}$ and $\Theta_\mu = \sum_{\nu \in E^{<mn}, [\nu]_n = \mu} \theta_{mn,\nu}$.
Straightforward calculations show that $(T, Q, \Theta)$ is a Toeplitz $n$-representation
of $E$, so the universal property of $\Tt(E, n)$ gives a homomorphism $i_{n,mn}$
satisfying the desired formulas. Using the formulas for the generators of $\Tt (E, n)$ in
Theorem~\ref{thm:TEn and CEn}, we see that for $\mu \in E^{<n}$,
\[
i_{n,mn} \Big(q_\mu - \sum_{(e,\tau) \in \mu E(n)^1} t_{(e,\tau)} t^*_{(e,\tau)}\Big)
    = \sum_{\nu \in E^{<mn}, [\nu]_n = \mu}
        \Big(q_{\nu} - \sum_{(e,\tau) \in \nu E(mn)^1} t_{(e,\tau)} t^*_{(e,\tau)}\Big).
\]
Theorem~4.1 of \cite{FowlerRaeburn:IUMJ99} implies that each term on the right hand side
of the preceding displayed equation is nonzero, and then also that $i_{n,mn}$ is
injective. Hence $i_{n,mn}$ is also injective.

For the final statement, observe that $i_{n,mn}$ clearly preserves the Cuntz--Krieger
relation, so it descends to a homomorphism $\tilde{i}_{n,mn} : C^*(E,n) \to C^*(E, mn)$.
A routine application of the gauge-invariant uniqueness theorem
\cite[Theorem~2.1]{BatesPaskEtAl:NYJM00} for $C^*(E(n))$ shows that $\tilde{i}_{n,mn}$ is
injective.
\end{proof}

Using the homomorphisms of the preceding proposition, we can form the direct limits
$\varinjlim \Tt(E, n_k)$ and $\varinjlim C^*(E,n_k)$. We write $i_{n_k,n_l} : \Tt(E, n_k)
\to \Tt(E, n_l)$ for the connecting homomorphism with $k < l$, and we write $i_{n_k,
\infty} : \Tt(E, n_k) \to \varinjlim \Tt(E, n_k)$ for the canonical inclusion. We will
also use these same symbols to denote the corresponding maps in the direct system
associated to the $C^*(E, n_k)$.

Fix a directed graph $E$. For $m,n \in \NN \setminus\{0\}$ such that $m \mid n$, we
define $p_{n,m} : E^{<n} \to E^{<m}$ by $p_{n,m}(\nu) = [\nu]_m$. Consider a sequence
$(n_k)^\infty_{k=1}$ such that $n_k \mid n_{k+1}$ for all $k$. The projective limit
$(\varprojlim E^{<n_k}, p_{n_{k+1}, n_k})$ can be realised as the topological subspace
\[
\Big\{(\mu_k)^\infty_{k=1} \in \prod^\infty_{k=1} E^{<n_k} :
    \mu_k = [\mu_{k+1}]_{n_k}\text{ for all $k \in \NN$}\Big\}.
\]

For a sequence $(n_k)^\infty_{k=1}$ as above and a directed graph $E$, given $k \in \NN$
and $\mu \in E^{<n_k}$, we write $Z(\mu,k)$ for the cylinder set $\{(\nu_i)^\infty_{i=1}
\in \varprojlim E^{<n_k} : \nu_k = \mu\}$. Observe that the $Z(\mu,k)$ are the canonical
compact open basis sets for the projective limit space regarded as a subspace of the
infinite product $\prod^\infty_{k=1} E^{<n_k}$. We write $\chi_{Z(\mu,k)}$ for the
characteristic function of $Z(\mu,k) \subseteq \varprojlim E^{<n_k}$.

\begin{dfn}\label{def:omega-rep}
Let $E$ be a row-finite directed graph with no sources, and suppose that $\omega =
(n_k)^\infty_{k=1}$ is a sequence of nonzero natural numbers such that $n_k \mid n_{k+1}$
for all $k$. A \emph{Toeplitz $\omega$-representation of $E$} is a triple $(T, Q, \psi)$
consisting of a Toeplitz--Cuntz--Krieger $E$-family in a $C^*$-algebra $B$ and a
homomorphism $\psi : C_0(\varprojlim E^{<n_k}) \to B$ such that $Q_w =
\psi(\chi_{Z(w,1)})$ for all $w \in E^0$, and
\[
T^*_e \psi(\chi_{Z(\mu,k)})
    = \begin{cases}
        \psi(\chi_{Z(\mu', k)})T^*_e &\text{ if $\mu = e\mu'$}\\
        \sum_{e\lambda \in E^{n_k}} \psi(\chi_{Z(\lambda,k)}) T^*_e &\text{ if $\mu = r(e)$} \\
        0 &\text{ otherwise}
    \end{cases}
\]
for all $e \in E^1$, $k \in \NN$ and $\mu \in E^{<n_k}$. If the pair $(T,Q)$ is a
Cuntz--Krieger $E$-family, then we call $(T, Q, \psi)$ a \emph{Cuntz--Krieger
$\omega$-representation}, or just an \emph{$\omega$-representation} of $E$.
\end{dfn}

We show that the universal $C^*$-algebra generated by an $\omega$-representation
coincides with Kribs and Solel's algebra $\varinjlim C^*(E(n_k))$. We first need a
multiplication formula analogous to that of Lemma~\ref{lem:mpctn formula}. To lighten
notation a bit, given a homomorphism $\psi : C_0(\varprojlim E^{<n_k}) \to B$, we will
write $\psi_{(\mu,k)}$ for the image of $\chi_{Z(\mu,k)}$ under $\psi$, which is a projection in $B$.

\begin{lem}\label{lem:omega mpctn}
Let $E$ be a row-finite directed graph with no sources, and let $\omega =
(n_k)^\infty_{k=1}$ be a sequence of nonzero natural numbers such that $n_k \mid n_{k+1}$
for all $k$. Let $(T, Q, \psi)$ be a Toeplitz $\omega$-representation of $E$. For
$\alpha,\beta,\gamma, \delta \in E^*$, $k \ge 1$ and $\mu,\nu \in E^{<n_k}$, we have
\[
(T_\alpha \psi_{(\mu,k)} T^*_\beta)(T_\gamma \psi_{(\nu,k)} T^*_\delta)
    = \begin{cases}
        T_\alpha \psi_{(\mu,k)} T^*_{\delta\beta'}
            &\text{ if $\beta = \gamma\beta'$ and $\nu = \beta'\mu$} \\
        T_\alpha \psi_{(\mu,k)} T^*_{\delta\nu\rho}
            &\text{ if $\beta = \gamma\nu\rho$ with $|\rho\mu| \in n\NN$}\\
        T_{\alpha\gamma'} \psi_{(\nu,k)} T^*_\delta
            &\text{ if $\gamma = \beta\gamma'$ and $\mu = \gamma'\nu$} \\
        T_{\alpha\mu\rho} \psi_{(\nu,k)} T^*_\delta
            &\text{ if $\gamma = \beta\mu\rho$ with $|\rho\nu| \in n\NN$}\\
        0   &\text{ otherwise.}
    \end{cases}
\]
In particular, $C^*(T, Q, \psi) = \clsp\{T_\alpha \psi_{(\mu,k)} T^*_\beta : k \ge 1, \mu
\in E^{<n_k}, \alpha,\beta \in E^*r(\mu)\}$.
\end{lem}
\begin{proof}
The first statement follows from the observation that each $(T, Q, \psi_{(\cdot, k)})$ is
a Toeplitz $n_k$-representation, and Lemma~\ref{lem:mpctn formula}. For the second
statement, first observe that the set on the right-hand side contains each $T_\alpha =
\sum_{\mu \in s(\alpha)E^{<n_1}} T_\alpha \psi_{(\mu, 1)} T^*_{s(\alpha)}$, each $Q_v =
\sum_{\mu \in vE^{<n_1}} T_v \psi_{(\mu, 1)} T^*_{v}$ and each $\psi_{(\mu,k)} =
T_{r(\mu)} \psi_{(\mu,k)} T^*_{r(\mu)}$. It is clearly closed under adjoints. So it
suffices to show that it is closed under multiplication. To see this, we consider a
product $T_\alpha \psi_{(\mu,k)} T^*_\beta T_\gamma \psi_{(\nu,l)} T^*_\delta$. Suppose
that $k \ge l$ (the case where $k < l$ will follow by taking adjoints). Then $Z(\nu,l) =
\bigsqcup_{\lambda \in E^{<n_k}, [\lambda]_{n_l} = \nu} Z(\lambda, k)$, and so we have
\[
T_\alpha \psi_{(\mu,k)} T^*_\beta T_\gamma \psi_{(\nu,l)} T^*_\delta
    = \sum_{\lambda \in E^{<n_k}, [\lambda]_{n_l} = \nu}
        T_\alpha \psi_{(\mu,k)} T^*_\beta T_\gamma \psi_{(\lambda,k)} T^*_\delta,
\]
and this belongs to $\clsp\{T_\alpha \psi_{(\mu,k)} T^*_\beta : k \ge 1, \mu \in
E^{<n_k}, \alpha,\beta \in E^*r(\mu)\}$ by the first statement.
\end{proof}

\begin{thm}\label{thm:omega universal}
Let $E$ be a row-finite directed graph with no sources, and let $\omega =
(n_k)^\infty_{k=1}$ be a sequence of nonzero natural numbers such that $n_k \mid n_{k+1}$
for all $k$. There is a Toeplitz $\omega$-representation $(t,q, \pi)$ of $E$ in
$\varinjlim \Tt(E, n_k)$ such that
\[
t_e = i_{n_1,\infty}(t_{n_1,e}),\qquad q_v = i_{n_1, \infty}(q_{n_1,v}),\qquad\text{ and }\quad
    \pi_{(\mu,k)} = i_{n_k,\infty}(\theta_{n_k, \mu})
\]
for all $e \in E^1$, all $v \in E^0$, and all $k \in \NN$ and $\mu \in E^{<n_k}$. This
Toeplitz $\omega$-representation is universal in the sense that if $(T, Q, \psi)$ is a
Toeplitz $\omega$-representation of $E$ in a $C^*$-algebra $B$, then there is a
homomorphism $\varphi_{T, Q, \psi} : \varinjlim \Tt(E, n_k) \to B$ such that
\[
\varphi_{T, Q, \psi}(t_e) = T_e,\qquad
    \varphi_{T, Q, \psi}(q_v) = Q_v,\quad\text{ and }\quad
    \varphi_{T, Q, \psi} \circ \pi = \psi.
\]
\end{thm}
\begin{proof}
We assume without loss of generality that $n_1 = 1$. The collection $(t_{n_1}, q_{n_1})$
is a Toeplitz--Cuntz--Krieger $E$-family and since $i_{n_1,\infty}$ is a homomorphism, it
follows that $t_e := i_{n_1,\infty}(t_{n_1,e})$ and $q_v = i_{n_1, \infty}(q_{n_1,v})$ is
a Toeplitz--Cuntz--Krieger $E$-family in $\varinjlim \Tt(E, n_k)$. For each $k$, the
formula
\begin{equation}\label{eq:pik}
\pi_k(\chi_{Z(\mu,k)}) := i_{n_k,\infty}(\theta_{n_k, \mu})
\end{equation}
gives a homomorphism $\pi_k : \lsp\{\chi_{Z(\mu,k)} : \mu \in E^{<n_k}\} \to \varinjlim
\Tt(E, n_k)$. So the universal property of $C_0(\varprojlim E^{<n_k}) \cong \varinjlim
C_0(E^{<n_k})$ yields a homomorphism $\pi : C_0(\varprojlim E^{<n_k}) \to \varinjlim
\Tt(E, n_k)$ satisfying $\pi_{(\mu,k)} = i_{n_k,\infty}(\theta_{n_k, \mu})$.

We check that $(t, q, \pi)$ is a Toeplitz $\omega$-representation. Since $n_1 = 1$, for
$w \in E^0$, we have $q_w = i_{n_1,\infty}(q_{n_1,w}) = i_{n_1,\infty}(\theta_{n_1,w}) =
\pi_{Z(w,1)}$. Take $e \in E^1$ and $\mu \in E^{<n_k}$. Then
\begin{align*}
t^*_e \pi_{(\mu,k)}
    &= i_{n_1,\infty}(t^*_{n_1,e}) i_{n_k,\infty}(\theta_{n_k,\mu})
     = i_{n_k,\infty}(i_{n_1,n_k}(t^*_{n_1,e}) \theta_{n_k,\mu}) \\
    &= i_{n_k,\infty}(t^*_{n_k,e} \theta_{n_k,\mu})
     = \begin{cases}
        i_{n_k,\infty}(\theta_{n_k,\mu'} t^*_{n_k,e}) &\text{ if $\mu = e\mu'$} \\
        i_{n_k,\infty}(\sum_{e\lambda \in E^{n_k}} \theta_{n_k,\lambda} t^*_{n_k,e}) &\text{ if $\mu = r(e)$} \\
        0 &\text{otherwise}
     \end{cases}\\
    &= \begin{cases}
        \pi_{(\mu',k)} t^*_e &\text{ if $\mu = e\mu'$} \\
        \sum_{e\lambda \in E^{n_k}} \pi_{(\lambda,k)} t^*_e &\text{ if $\mu = r(e)$} \\
        0 &\text{otherwise.}
    \end{cases}
\end{align*}
So $(t,q,\pi)$ is a Toeplitz $\omega$-representation of $E$ in $\varinjlim \Tt(E,n_k)$.

Let $(T, Q, \psi)$ be another $\omega$-representation of $E$ in $B$, and fix $k \in \NN$.
For $\mu \in E^{<n_k}$ let $\Theta_\mu := \psi_{(\mu,k)}$. Quick calculations show that
$(T, Q, \Theta)$ is a Toeplitz $n_k$-representation of $E$. The universal property of
$\Tt(E, n_k)$ gives a homomorphism $\varphi_{n_k, \infty} : \Tt(E, n_k) \to B$ satisfying
\[
\varphi_{n_k, \infty}(t_e) = T_e,\qquad
    \varphi_{n_k, \infty}(q_v) = Q_v,\quad\text{ and }\quad
    \varphi_{n_k, \infty}(\theta_{n_k, \mu}) = \psi_{(\mu,k)}.
\]
We check that $\varphi_{n_{k+1}, \infty} \circ i_{n_k, n_{k+1}} = \varphi_{n_k,\infty}$.
We have
\[
\varphi_{n_{k+1},\infty} \circ i_{n_k, n_{k+1}}(t_{n_k, e}) =
\varphi_{n_{k+1},\infty}(t_{n_{k+1}, e}) = T_e = \varphi_{n_k, \infty}(t_{n_k, e}),
\]
and similarly $\varphi_{n_{k+1}} \circ i_{n_k, n_{k+1}}(q_{n_k,v}) = Q_v =
\varphi_{n_k}(q_{n_k,v})$. For $\mu \in E^{<n_k}$,
\begin{align*}
\varphi_{n_{k+1}, \infty}(i_{n_k, n_{k+1}}(\theta_{n_k, \mu}))
    &= \varphi_{n_k, \infty}\Big(\sum_{\lambda \in E^{<n_{k+1}}, [\lambda]_{n_k} = \mu} \theta_{n_{k+1}, \lambda}\Big)\\
    &= \psi\Big(\sum_{\lambda \in E^{<n_{k+1}}, [\lambda]_{n_k} = \mu} \chi_{Z(\lambda,n_{k+1})}\Big)
    = \psi(\chi_{Z(\mu,k)}) = \varphi_{n_k, \infty}(\theta_{n_k, \mu}).
\end{align*}
The universal property of $\varinjlim \Tt(E, n_k)$ now gives a homomorphism $\varphi_{T,
Q, \psi}$ making the diagram
\[\begin{tikzpicture}[xscale=2]
    \node (k) at (0,3) {$\Tt(E, n_k)$};
    \node (k+1) at (4,3) {$\Tt(E, n_{k+1})$};
    \node (lim) at (2,1.5) {$\varinjlim \Tt(E, n_k)$};
    \node (B) at (2,-0.5) {$B$};
    \draw[-latex] (k)--(k+1) node[pos=0.5, above]{$i_{n_k, n_{k+1}}$};
    \draw[-latex] (k)--(lim) node[pos=0.75, above]{$i_{n_k, \infty}$};
    \draw[-latex] (k+1)--(lim) node[pos=0.75, above]{$i_{n_{k+1}, \infty}$};
    \draw[-latex] (k)--(B) node[pos=0.5, inner sep=0pt, anchor=north east]{$\varphi_{n_k,\infty}$};
    \draw[-latex] (k+1)--(B) node[pos=0.5, inner sep=0pt, anchor=north west]{$\varphi_{n_{k+1},\infty}$};
    \draw[-latex] (lim)--(B) node[pos=0.25, inner sep=0.5pt, right]{$\varphi_{T,Q,\psi}$};
\end{tikzpicture}\]
commute, and this homomorphism has the desired properties.
\end{proof}

Given $E$ and $\omega$ as in Theorem~\ref{thm:omega universal}, we write $\Tt(E, \omega)$
for the universal $C^*$-algebra generated by a Toeplitz $\omega$-representation of $E$.
Since the universal $C^*$-algebra for a given set of generators and relations is unique
up to canonical isomorphism, we can and will identify $\Tt(E, \omega)$ with $\varinjlim
\Tt(E(n_k))$ via the homomorphism of Theorem~\ref{thm:omega universal}.

The following theorem follows from the same argument as Theorem~\ref{thm:omega
universal}.

\begin{thm}\label{thm:universal omegarep}
Let $E$ be a row-finite directed graph with no sources, and let $\omega =
(n_k)^\infty_{k=1}$ be a sequence of nonzero natural numbers such that $n_k \mid n_{k+1}$
for all $k$. There is an $\omega$-representation $(s, p, \rho)$ of $E$ in $\varinjlim
C^*(E, n_k)$ such that
\[
s_e = i_{n_1,\infty}(s_{n_1,e}),\qquad p_v = i_{n_1, \infty}(p_{n_1,v}),\qquad\text{ and }\quad
    \rho_{(\mu,k)} = i_{n_k,\infty}(\varepsilon_{n_k, \mu})
\]
for all $e \in E^1$, all $v \in E^0$, and all $k \in \NN$ and $\mu \in E^{<n_k}$. This
$\omega$-representation is universal in the sense that if $(S, P, \psi)$ is an
$\omega$-representation of $E$ in a $C^*$-algebra $B$, then there is a homomorphism
$\varphi_{S, P, \psi} : \varinjlim C^*(E, n_k) \to B$ such that
\[
\varphi_{S, P, \psi}(s_e) = S_e,\qquad
    \varphi_{S, P, \psi}(p_v) = P_v,\quad\text{ and }\quad
    \varphi_{S, P, \psi} \circ \rho = \psi.
\]
\end{thm}

We write $C^*(E, \omega)$ for the universal $C^*$-algebra generated by an
$\omega$-representation of $E$, and we identify it with $\varinjlim C^*(E(n))$ via the
homomorphism of the preceding theorem.

Kribs and Solel regard $\varinjlim C^*(E, n_k)$ as a generalised Bunce--Deddens algebra.
Since the Bunce--Deddens algebra $B_\omega$ is completely determined by the supernatural
number $\omega$, we expect $C^*(E, \omega)$ to depend only on $E$ and the supernatural
number associated to $\omega$. We give an elementary proof that this is the case using
the presentation given in Theorem~\ref{thm:universal omegarep}. For this, recall that for
sequences $\omega = (n_k)^\infty_{k=1}$ with $n_k \mid n_{k+1}$ for all $k$, and $\omega'
= (m_l)^\infty_{l=1}$ with $m_l \mid m_{l+1}$ for all $l$, we write $\omega \mid \omega'$
if for every $k \ge 1$ there exits $j(k) \ge 1$ such that $n_k \mid m_{j(k)}$. The
supernatural number $[\omega]$ associated to $\omega$ is the collection $[\omega] :=
\{\omega' : \omega \mid \omega'\text{ and }\omega' \mid \omega\}$.

\begin{prp}\label{prp:equiv omegas}
Let $E$ be a row-finite directed graph with no sources. Let $\omega = (n_k)^\infty_{k=1}$
and $\omega' = (m_j)^\infty_{j=1}$ be sequences of nonzero natural numbers such that $n_k
\mid n_{k+1}$ for all $k$ and $m_j \mid m_{j+1}$ for all $j$. If $\omega \mid \omega'$,
then there is an injective homomorphism $\varphi_{\omega,\omega'} : \Tt(E, \omega) \to
\Tt(E, \omega')$ such that
\begin{equation}\label{eq:omega->omega'}
\varphi_{\omega, \omega'} \circ i_{n_k, \infty}
    = i_{m_{j(k)}, \infty} \circ i_{n_k, m_{j(k)}}\quad
        \text{ for all $k \ge 1$ and any $j(k)$ such that $n_k \mid m_{j(k)}$.}
\end{equation}
Moreover, $\varphi_{\omega, \omega'}$ descends to a homomorphism $\tilde\varphi_{\omega,
\omega'} : C^*(E, \omega) \to C^*(E, \omega')$. If $[\omega] = [\omega']$ then
$\varphi_{\omega, \omega'} : \Tt(E, \omega) \to \Tt(E, \omega')$, and
$\tilde\varphi_{\omega,\omega'} : C^*(E, \omega) \to C^*(E, \omega')$ are isomorphisms.
\end{prp}
\begin{proof}
Fix natural numbers $j(k)$ such that $n_k \mid m_{j(k)}$ for all $k$. Then $i_{m_{j(k)},
\infty} \circ i_{n_k, m_{j(k)}} : \Tt(E, n_k) \to \varinjlim \Tt(E, m_k)$ is a
homomorphism for each $k$. Since
\begin{align*}
i_{m_{j(k+1)}, \infty} \circ i_{n_{k+1}, m_{j(k+1)}} \circ i_{n_k, n_{k+1}}
    &= i_{m_{j(k+1)}, \infty} \circ i_{n_k, m_{j(k+1)}}\\
    &= i_{m_{j(k+1)}, \infty} \circ i_{m_{j(k)}, m_{j(k+1)}} \circ i_{n_k, m_{j(k)}}
     = i_{m_{j(k)}, \infty} \circ i_{n_k, m_{j(k)}},
\end{align*}
The universal property of $\varinjlim \Tt(E, n_k)$ gives a homomorphism $\varphi$ that
satisfies~\eqref{eq:omega->omega'}. The same argument shows that $\varphi$ descends to a
homomorphism $\tilde\varphi : \varinjlim C^*(E, n_k) \to \varinjlim C^*(E, m_l)$. Now
suppose that $\omega' \mid \omega$ as well. The preceding paragraph gives a homomorphism
$\gamma : \Tt(E, \omega') \to \Tt(E, \omega)$ such that $\gamma \circ i_{m_j, \infty} =
i_{n_{k(j)}, \infty} \circ i_{m_j, n_{k(j)}}$ for all $j$, and which descends to
$\tilde\gamma : C^*(E, \omega') \to C^*(E, \omega)$. It is routine to check that $\gamma
\circ \phi$ is the identity map on each $i_{n_k, \infty} \Tt(E, n_k)$ and symmetrically,
$\phi \circ \gamma$ is the identity on each $i_{m_j}(\Tt(E, m_j))$, so continuity shows
that $\phi$ and $\gamma$ are mutually inverse; the same argument shows that
$\tilde{\varphi}$ and $\tilde{\gamma}$ are mutually inverse.
\end{proof}

\section{The topological graph \texorpdfstring{$E(\infty)$}{E(infty)}}\label{sec:topgraph}

Kribs and Solel construct a topological graph $E(\infty)$ from a graph $E$ and a
supernatural number $\omega$. They show in \cite[Theorem~6.3]{KribsSolel:JAMS07} that
$C^*(E, \omega)$ is isomorphic to the $C^*$-algebra $C^*(E(\infty))$ of this topological
graph in the sense of Katsura \cite{KatsuraI}. Unfortunately, their statement does not
give explicit details about the isomorphism, and we shall need these in the sequel. In
this section, we give a slightly different description of the topological graph
$E(\infty)$, and use it to present the details of the isomorphism $C^*(E,\omega) \cong
C^*(E(\infty))$. For the most part, we are just making explicit some of the details of
the proofs of results in \cite{KribsSolel:JAMS07} and \cite{KatsuraII}, so we keep our
presentation short.

First recall that a topological graph $F$ consists of second-countable locally compact
Hausdorff spaces $F^0$ and $F^1$ and maps $r, s : F^1 \to F^0$ such that $r$ is
continuous and $s$ is a homeomorphism. Katsura \cite{KatsuraI} associates to each
topological graph $F$ a $C^*$-algebra that we denote $C^*(F)$. This $C^*(F)$ is generated
by a homomorphism $t_F^0 : C_0(F^0) \to C^*(F)$ and a linear map $t_F^1 : C_c(F^1) \to
C^*(F)$ satisfying relations reminiscent of the Cuntz--Krieger relations for graph
algebras (for a description that avoids the machinery of Hilbert modules, see
\cite{LiPaskSims:NYJM14}). The pair $(t_F^0, t_F^1)$ is called a Cuntz--Krieger $E$-pair.
When $F^0$ and $F^1$ are discrete and countable, $C^*(F)$ coincides with the usual graph
$C^*$-algebra described in Section~\ref{sec:background}.

Now let $E$ be a row-finite directed graph with no sources, and take a sequence $\omega =
(n_k)^\infty_{k=1}$ of nonzero positive integers such that $n_k \mid n_{k+1}$ for all
$k$. Suppose that $n_k \to \infty$ as $k \to \infty$. Let $X_i = \{w \in E^*: 0 \leq |w|
< n_i, |w| \equiv 0 \, (\operatorname{mod}\,n_{i-1}) \}$, let $X = \Pi_{i=1}^\infty X_i$
and let $Y = \{y \in X: s(y_k) = r(y_{k+1}) \}$. For each $e \in E^1$, let $D_e = \{y \in Y:
r(y_1) = s(e) \}$ and $R_e = \{y \in Y: \text{ for some } l \leq \infty, y_i = r(e)
\text{ for all } i < l \text{ and (if $l \not = \infty$) } y_l = e y' \text{ for some }
|y'| \equiv -1 \, (\operatorname{mod}\,n_{i-1}) \}$. For $y \in D_e$, write $i(y)$ for
the smallest positive integer such that $|y_i| < n_i - n_{i-1}$ or $i(y) = \infty$ if
$|y_i| = n_i - n_{i-1}$ for every $i$. If $i(y) < \infty$, write $\sigma_e(y) = u$, where
\[
    u_i = \begin{cases} r(e) & \text{ if } i < i(y) \\
        e y_1 \dots y_{i(y)} & \text{ if } i = i(y) \\
        y_i & \text{ if } i > i(y).
    \end{cases}
\]
If $i(y) = \infty$, set $\sigma_e(y) = (r(e), r(e) \dots)$.

Kribs and Solel construct a topological graph $E(\infty)$ with $E(\infty)^0 = Y$,
$E(\infty)^1 = \{(e, y) \in E^1 \times Y: y \in D_e \}$, $s_{E(\infty)}(e, y) = y$ and
$r_{E(\infty)}(e, y) = \sigma_e(y)$. Here we give another presentation of $E(\infty)$
which is more natural within our framework.

\begin{lem}\label{lem:EisoF}
Let $E$ be a row-finite directed graph with no sources, and take a sequence $\omega =
(n_k)^\infty_{k=1}$ of nonzero positive integers such that $n_k \mid n_{k+1}$ for all
$k$. Suppose that $n_k \to \infty$ as $k \to \infty$. Define $F_{E, \omega}^0 =
\varprojlim E^{<n_k}$, $F_{E, \omega}^1 = \{(e,x) \in E^1 \times \varprojlim E^{<n_k}:
r(x_1) = s(e) \}$, $s_F(e,x) = x$, and $r_F(e,x)_k = r_{n_k}(e,x_k) = [e x_k]_{n_k}$.
Then $\phi = (\phi^0,\phi^1): E(\infty) \to F$ defined  by $\phi^0(y)_i = y_1 y_2 \dots
y_i$ for $y \in E(\infty)^0$ and $\phi^1(e, y)_i = (e, \phi^0(y))$ for $(e, y) \in
E(\infty)^1$ is an isomorphism of topological graphs.
\end{lem}
\begin{proof}
We abbreviate $F^i := F_{E,\omega}^i$, $i = 0,1$. Define $\psi = (\psi^0,\psi^1): F \to
Y$ by $\psi^0(x)_1 = x_1$ and $[x_{i+1}]_{n_i} \psi^0(x)_{i+1} = x_{i+1}$ for all $i \geq
1$, and $\psi^1(e, x) = (e, \psi^0(x))$. We have $\psi^0(\phi^0(y))_i = \psi^0((y_1 \dots
y_j)_{j=1}^\infty)_i$. Since $|y_1 \dots y_{i-1}| = \sum_{j=1}^{i-1} |y_j| < n_{i-1}$
and $|y_i| \in n_{i-1} \mathbb{N}$, we have $[y_1 \dots y_i]_{n_{i-1}} = y_1 \dots y_{i-1}$ and
hence $\psi^0(\phi^0(y))_i = y_i$. Conversely, $\phi^0(\psi^0(x))_i = \psi^0(x)_1 \dots
\psi^0(x)_i = x_1 \psi^0(x)_1 \dots \psi^0(x)_i = x_2 \psi^0(x_3) = \dots = x_i$.
Therefore $\psi^0$ is an inverse for $\phi^0$.

The basic open sets in $Y$ are given by $Z_Y(w_1, \dots, w_k) = \{ y \in Y: y_i = w_i
\text{ for } 1 \leq i \leq k \}$, where $w_i \in X_i$ and $s(w_i) = r(w_{i+1})$.

We calculate \begin{align*} \phi^0(Z_Y(y_1, \dots, y_k)) &= \{x \in F^0: x_i = y_1 \dots y_i \text{ for all } i \leq k \} \\
&= \{x \in F^0: x_k = y_1 \dots y_k \} = Z(y_1 \dots y_k, k).  \end{align*} So $\psi^0$
is continuous.

Conversely, for $\mu \in E^{<n_k}$, express $\mu = y_1 \dots y_k$, where $y_1 =
[\mu]_{n_1}$ and $[\mu]_{n_i} y_{i+1} = [y]_{n_{i+1}}$. Then $\psi^0(Z(\mu, k)) =
Z_Y(y_1, \dots, y_k)$. So $\phi^0$ is continuous. Therefore $\phi^0$ is a homeomorphism of
$E(\infty)^0$ onto $F^0$. It then follows immediately that $\phi^1: E(\infty)^1 \to F^1$ is also
a homeomorphism.

We have $\phi^0(s_{E(\infty)}(e, y)) = s_F(e, \phi^0(y)) = s_F(\phi^1(e,y))$. We also
have $\phi^0(r_{E(\infty)}(e, y))_i = [e y_1 \dots y_i]_{n_i}$, so
$\phi^0(r_{E(\infty)}(e, y)) = r_F(e, \phi^0(y)) = r_F(\phi^1(e,y))$. Therefore $\phi$ is
an isomorphism of topological graphs.
\end{proof}

We now analyse connectivity in the topological graph $F$ when $E$ is
finite and strongly connected.

We need to recall some facts from Perron-Frobenius theory for finite strongly connected
graphs. Recall (for example from \cite[Section~6]{LacaLarsenEtAl:xx14} with $k=1$) that
the \emph{period} $\Pp_E$ of a strongly connected directed graph $E$ is given by $\Pp_E =
\gcd\{|\mu| : \mu \in E^*, r(\mu) = s(\mu)\}$. The group $\Pp_E\ZZ$ is then equal to the
subgroup generated by $\{|\mu| : \mu \in v E^* v\}$ for any vertex $v$ of $E$, and so is
equal to $\{|\mu| - |\nu| : \mu,\nu \in v E^* v\}$ for any $v$.

\begin{lem}\label{lem:Cn}
Let $E$ be a strongly connected finite graph with no sources, and take $n \in \NN$. There
is a map $C_n : E^0 \times E^0 \to \ZZ/\gcd(\Pp_E, n)\ZZ$ such that $C_n(r(\lambda),
s(\lambda)) = |\lambda| + \gcd(\Pp_E, n)\ZZ$ for all $\lambda \in E^*$. There is also an
equivalence relation $\sim_n$ on $E^0$ such that $v \sim_n w$ if and only if $C_n(v,w) =
0$.
\end{lem}
\begin{proof}
Fix $v,w \in E^0$ and $\mu,\nu \in v E^* w$. Since $E$ is strongly connected, there is a
path $\lambda \in w E^* v$, and then $\mu\lambda, \nu\lambda \in vE^* v$. Hence $|\mu| -
|\nu| = |\mu\lambda| - |\nu\lambda| \in \Pp_E\ZZ \subseteq \gcd(\Pp_E, n)\ZZ$. So there
is a well-defined function $C_n : \{(v,w) \in E^0 \times E^0 : v E^* w \not= \emptyset\}
\to \ZZ/\gcd(\Pp_E, n)\ZZ$ such that $C_n(r(\lambda), s(\lambda)) = |\lambda| +
\gcd(\Pp_E, n)\ZZ$ for all $\lambda$. Since $E$ is strongly connected, the domain of
$C_n$ is all of $E^0 \times E^0$ as claimed.

Define a relation $\sim_n$ on $E^0$ by $v\sim_n w$ if $C_n(v,w) = 0$. We show that
$\sim_n$ is an equivalence relation. We clearly have $C_n(v,v) = 0$ for all $v$, so
$\sim_n$ is reflexive. To see that it is symmetric, suppose that $C_n(v,w) = 0$. Then
there exists $\lambda \in vE^* w$ with $|\lambda| \in \gcd(\Pp_E, n)\ZZ$. Since $E$ is
strongly connected, there exists $\mu \in w E^* v$, and then $\lambda\mu \in v E^* v$.
Hence $|\lambda\mu| \in \Pp_E\ZZ$. Now $|\mu| = |\lambda\mu| - |\lambda| \in \Pp_E \ZZ
\subseteq \gcd(\Pp_E, n)\ZZ$, and so $C_n(w,v) = 0$ as well. For transitivity,
suppose that $C_n(u,v) = 0$ and $C_n(v,w) = 0$. Then there exist $\mu \in uE^* v$ and
$\nu \in v E^* w$ with $|\mu|,|\nu| \in \gcd(\Pp_E, n)\ZZ$. So $\mu\nu \in uE^* w$
satisfies $|\mu\nu| = |\mu| + |\nu| \in \gcd(\Pp_E, n)\ZZ$, and hence $C_n(u,w) = 0$ too.
\end{proof}

\begin{prp}\label{prp:components}
Let $E$ be a strongly connected finite directed graph with no sources. For $n \in \NN$,
the connected components of $E(n)$ are the sets $E(n)^0_\Lambda := \{\mu \in E^{<n} :
s(\mu) \in \Lambda\}$ indexed by $\Lambda \in E^0/{\sim_n}$. These connected components
are all strongly-connected: if $\mu,\nu \in E(n)^0_\Lambda$, then $\mu E(n)^* \nu \not=
\emptyset$. In particular, $E(n)$ is strongly connected if and only if $\gcd(\Pp_E, n) =
1$.
\end{prp}

Recall that for $\lambda = \lambda_1 \dots \lambda_l \in E^*$ and $\mu \in E^{<n}$ with
$s(\lambda) = r(\mu)$, we write $(\lambda,\mu)$ for the corresponding path $(\lambda_1,
[\lambda_2\dots\lambda_l\mu]_n)(\lambda_2,[\lambda_3\dots\lambda_l\mu]_n) \dots
(\lambda_l,\mu) \in [\lambda\mu]_n E(n)^l \mu$. In particular, if $\lambda \in E^l$, then
$(\lambda, s(\lambda)) \in E(n)^l$.

We write $\approx_E$ for the smallest equivalence relation on $E^0$ such that $r(e)
\approx_E s(e)$ for all $e \in E^1$. We call the equivalence classes of $\approx_E$ the
connected components of $E$.

\begin{proof}[Proof of Proposition~\ref{prp:components}]
Since $\mu E(n)^* \nu \not= \emptyset$ implies $\mu \approx_{E(n)} \nu$, it suffices to
show that if $s(\mu) \sim_n s(\nu)$ then $\mu E(n)^* \nu \not= \emptyset$, and that if
$\mu \approx_{E(n)} \nu$, then $s(\mu) \sim_n s(\nu)$.

First suppose that $s(\mu) \sim_n s(\nu)$. Since $E$ has no sources and is strongly
connected, it has no sinks, so we can choose $\alpha = \alpha_1 \dots \alpha_k \in E^*
r(\nu)$ such that $|\alpha\nu| \in n\NN$. It follows that $C(r(\alpha), s(\nu)) = 0$ and
so $s(\mu) \sim_n r(\alpha)$. Let $v := s(\mu)$ and $w := r(\alpha)$. Since $v \sim_n w$,
we have $|\lambda| + \gcd(\Pp_E, n)\ZZ = C_n(v,w) = 0$, so $|\lambda| \in \gcd(\Pp_E,
n)\ZZ$. Choose $k$ such that $k \Pp_E \equiv \gcd(\Pp_E, n)\;(\operatorname{mod}\,n)$.
Since $E$ is strongly connected, we have $\Pp_E \ZZ = \{|\eta| - |\zeta| : \eta,\zeta \in
wE^* w\}$. So there are cycles $\eta, \zeta \in wE^*w$ such that $|\eta| - |\zeta| =
\Pp_E$. In particular, $|\eta\zeta^{n-1}| = |\eta| - |\zeta| + |\zeta^n| = \Pp_E +
n|\zeta| \equiv \Pp_E\;(\operatorname{mod}\,n)$. Hence $\beta := (\eta\zeta^{n-1})^k \in
w E^* w$ satisfies $|\beta| \equiv k\Pp_E\;(\operatorname{mod}\,n) \equiv \gcd(\Pp_E,
n)\;(\operatorname{mod}\,n)$. Choose $q \in \NN$ such that $qn \ge |\lambda|$. Since
$|\lambda|$ is divisible by $n$, the number $l := \frac{qn - |\lambda|}{\gcd(\Pp_E, n)}$
is an integer. Now $|\lambda \beta^l| \in v E^{jn} w$ for some $j$.
So~\eqref{eq:connectivity in E(n)} gives a path $\tilde\lambda \in v E(n)^* w$. Now
$(\mu, \nu) \tilde\lambda (\alpha,\nu) \in \mu E(n)\nu$ as required.

Now suppose that $\mu \approx_{E(n)} \nu$. Since $(\mu,s(\mu)) \in \mu E(n)^* s(\mu)$
and likewise for $\nu$, and since
$\approx_{E(n)}^0$ is an equivalence relation, we have $s(\mu) \approx_{E(n)} s(\nu)$. So
it suffices to show that $v \approx_{E(n)} w$ implies $v \sim_n w$ for $v,w \in E^0$. By
definition of $\approx_{E(n)}$ it then suffices, by induction, to show that if $v E(n)^*
w \not= \emptyset$, say $(\lambda, w) \in v E(n)^* w$, then $v \sim_n w$.
By~\eqref{eq:connectivity in E(n)} we have $\lambda \in v E^{jn} w$ for some $j$. In
particular, $C(v,w) = |\lambda| + \gcd(\Pp_E, n)\ZZ = 0 + \gcd(\Pp_E, n)\ZZ$ and so $v
\sim_n w$.
\end{proof}

Given a sequence $\omega = (n_k)^\infty_{k=1}$ of natural numbers with $n_k \mid n_{k+1}$
for all $k$, and given $p \in \NN$, the sequence $\gcd(p, n_k)$ is nondecreasing and
bounded above by $p$, so it is eventually constant. We write $\gcd(p, \omega)$ for its
eventual value.

\begin{lem}\label{lem:invariantsubset}
Let $E$ be a strongly connected finite directed graph with no sources, and take a
sequence $\omega = (n_k)^\infty_{k=1}$ of nonzero positive integers such that $n_k \mid
n_{k+1}$ for all $k$. Fix $k$ with $\gcd(\Pp_E, n_k) = \gcd(\Pp_E, \omega)$.  For each
equivalence class $\Lambda \in E^0 / \sim_{n_k}$, let $X_\Lambda = \bigcup_{\mu \in
E^{<n_k}, s(\mu) \in \Lambda} Z(\mu, k)$. The $X_\Lambda$ are mutually disjoint and cover
$F^0 = \varprojlim E^{<n_k}$. Each $X_\Lambda$ is invariant in the sense of
\cite[Definition~2.1]{KatsuraIII}, and the $X_\Lambda$ are the minimal nonempty closed
invariant subsets of $F^0$.
\end{lem}
\begin{proof}
Take $\Lambda, \Lambda' \in E^0 / \sim_{n_k}$ with $\Lambda \not = \Lambda'$. Since $x \in
X_\Lambda$ if and only if $s(x_k) \in \Lambda$, it is clear that $X_\Lambda$ and
$X_{\Lambda'}$ are mutually orthogonal.

To see that each $X_\Lambda$ is invariant, let $\mu \in E^{<n_k}$  and $e \in E^1
r(\mu)$. Then
\[
r_{n_k}(e,\mu) = [e \mu]_{n_{k}} =
    \begin{cases}
        e \mu & \text{ if $|e \mu | < n_{k}$} \\
        r(e) & \text{ if $|e \mu| = n_{k}$},
    \end{cases}
\]
so $s(r_{n_{k}}(e, \mu)) \in \Lambda$ if and only if $s(\mu) \in \Lambda$. Now, take
$(e,x) \in F^1$. We have $s(x_k) \in \Lambda$ if and only if $s(r_{n_k}(e,x_k)) =
s([e x_k]_{n_k}) \in \Lambda$. So $x \in X_\Lambda$ if and only if $r_F(e,x) = [e
x_k]_{n_k} \in X_\Lambda$.

For the final assertion, fix $x = (x_k)_{k=1}^\infty$ and $y = (y_k)_{k=1}^\infty$ in a
given $X_\Lambda$. It suffices to show that for every $k \in \NN$, there exists $\mu_k
\in F^*$ such that $s_F(\mu_k) =x$ and $r_F(\mu_k) \in Z(y_k,k)$. Fix $k$ such that
$\gcd(\Pp_E, n_k) = \gcd(\Pp_E, \omega)$. Proposition~\ref{prp:components} implies that
the component $E(n_k)_\Lambda^0$ is strongly connected. So there exists $\lambda \in
E(n_k)^*$ such that $s_{n_k}(\lambda) = x_k$ and $r_{n_k}(\lambda) = y_k$. Say $\lambda =
(\lambda_1, [\lambda_2 \dots \lambda_i x_k]_{n_k}) \dots (\lambda_{i-1}, [\lambda_i
x_k]_{n_k})(\lambda_i, x_k).$  Define $\mu_i := (\lambda_i,x) \in F^1$ and inductively
let $\mu_j = (\lambda_j, r_F(\mu_{j+1})) \in F^1$ for $1 \leq j \leq i-1$. Then $\mu =
\mu_1 \dots \mu_i \in F^i$ and $s_F(\mu) = s_F(\mu_i) = x$. By construction, $(\mu_j)_k =
(\lambda_j, [\lambda_{j+1}x_k]_{n_k})$ for each $1 \leq j \leq i-1$, so $r_F(\mu)_k =
r_F(\mu_1)_k = r_{n_k}(\lambda_1, [\lambda_2 \dots \lambda_i x_k]_{n_k}) =
r_{n_k}(\lambda) = y_k$, so $r_F(\mu) \in Z(y_k, k)$.
\end{proof}

\section{Uniqueness theorems and simplicity}\label{sec:uniqueness}

In this section we prove uniqueness theorems for $\Tt(E, \omega)$ and $C^*(E, \omega)$.
Interestingly, in contrast to the uniqueness theorems for directed graph algebras,
no gauge-invariance hypothesis or aperiodicity hypotheses are needed in
the uniqueness theorem for $C^*(E, \omega)$ provided that $n_k \to \infty$. To obtain our
uniqueness theorem for $C^*(E,\omega)$ we appeal to Katsura's theory of topological
graphs and their $C^*$-algebras using the construction of the preceding section.

Our first uniqueness theorem is for $\Tt(E, \omega)$, and follows relatively easily from
Fowler and Raeburn's uniqueness theorem \cite[Theorem~4.1]{FowlerRaeburn:IUMJ99} for
Toeplitz algebras of Hilbert bimodules.

\begin{prp}\label{prp:coburn}
Let $E$ be a row-finite directed graph with no sources, and take a sequence $\omega =
(n_k)^\infty_{k=1}$ of nonzero positive integers such that $n_k \mid n_{k+1}$ for all
$k$. Let $(T, Q, \psi)$ be an $\omega$-representation of $E$ in a $C^*$-algebra $A$. Then
the induced homomorphism $\pi_{T, Q, \psi} : \Tt(E, \omega) \to A$ is injective if and
only if
\begin{equation}\label{eq:Toeplitz injectivity}
\Big(Q_{r(\mu)} - \sum_{e \in r(\mu)E^1} T_e T^*_e\Big)\psi_{(\mu,k)} \not= 0
\end{equation}
for all $k \in \NN$ and $\mu \in E^{<n_k}$.
\end{prp}
\begin{proof}
Fix $k \in \NN$ and let $(t_{n_k, (e,\mu)}, q_{n_k, \mu})$ be the universal
Toeplitz--Cuntz--Krieger $E(n_k)$-family in $\Tt C^*(E(n_k)) = \Tt(E, n_k)$. Fix $\mu \in
E^{<n_k}$. The composition $\varphi_{T,Q,\psi} \circ i_{n_k,\infty}$ carries $q_{n_k,
\mu} - \sum_{(e,\nu) \in \mu E(n_k)^1} t_{n_k, (e,\nu)} t^*_{n_k, (e,\nu)}$ to
$\psi_{(\mu,k)} - \sum_{(e,\nu) \in \mu E(n_k)^1} T_e \psi_{(\nu,k)} T^*_e$. Applying the
relation~\eqref{eq:nrep rel} and collecting terms we obtain
\[
\varphi_{T,Q,\psi} \circ i_{n_k,\infty}
    \Big(q_{n_k, \mu} - \sum_{(e,\nu) \in \mu E(n_k)^1} t_{n_k, (e,\nu)} t^*_{n_k, (e,\nu)}\Big)
    = \Big(Q_{r(\mu)} - \sum_{e \in r(\mu)E^1} T_e T^*_e\Big)\psi_{(\mu,k)}.
\]
Theorem~4.1 of \cite{FowlerRaeburn:IUMJ99} shows that $\varphi_{T, Q, \psi} \circ i_{n_k,
\infty} : \Tt C^*(E(n_k)) \to A$ is injective if and only if $\varphi_{T, Q, \psi} \circ
i_{n_k, \infty}\Big(q_{n_k, \mu} - \sum_{(e,\nu) \in \mu E(n_k)^1} t_{n_k, (e,\nu)}
t^*_{n_k, (e,\nu)}\Big) \not= 0$ for all $\mu \in E^{<n_k}$. Since $i_{n_k,\infty}$ is
injective for each $k$, the result follows.
\end{proof}

We now state our main uniqueness result, which characterises the injective homomorphisms
of $C^*(E, \omega)$.

\begin{thm}\label{thm:CKUT}
Let $E$ be a row-finite directed graph with no sources, and take a sequence $\omega =
(n_k)^\infty_{k=1}$ of nonzero positive integers such that $n_k \mid n_{k+1}$ for all
$k$. Suppose that $n_k \to \infty$ as $k \to \infty$. Suppose that $(S, P, \psi)$ is an
$\omega$-representation of $E$. Then $\varphi_{S, P, \psi}$ is injective if and only if
$\psi_{(\mu,k)} \not= 0$ for all $k \in \NN$ and $\mu \in E^{<n_k}$.
\end{thm}

To prove this theorem, we use Katsura's results about topological graph $C^*$-algebras,
and the isomorphism $C^*(E, \omega) \cong C^*(E(\infty))$ established by Kribs and Solel.
The following result follows from the isomorphism $F \cong E(\infty)$, and Katsura's
arguments in \cite{KatsuraII}, but a precise description of the isomorphism that we need
to use is not provided there, so we give a detailed statement.

\begin{prp}\label{prp:IsomF}
Let $E$ be a row-finite directed graph with no sources, and take a sequence $\omega =
(n_k)^\infty_{k=1}$ of nonzero positive integers such that $n_k \mid n_{k+1}$ for all
$k$. Suppose that $n_k \to \infty$ as $k \to \infty$. There is an isomorphism $\pi :
\varinjlim C^*(E(n_k)) \to C^*(F)$ such that \[\pi(j_{n_k,\infty}(p_{n_k, \lambda})) =
t_F^0(\chi_{Z(\lambda,k)}), \quad \text{ and } \quad \pi(j_{n_k,\infty}(n_k,
s_{(e,\lambda)}))\ = t_F^1(\chi_{\{e\} \times Z(\lambda,k)}), \] where $(t_F^0,t_F^1)$ is
the universal Cuntz--Krieger pair for $C^*(F)$.
\end{prp}
\begin{proof}
For a topological graph $E$ we denote by $(t_E^0,t_E^1)$ the universal Cuntz--Krieger
pair for $C^*(E)$. The argument of Katsura \cite[Proposition~2.9]{KatsuraII} shows that
each regular factor map $m:E \to F$ of topological graphs $E$ and $F$ induces a
homomorphism $\mu_m:C^*(F) \to C^*(E)$ such that $\mu_m \circ t_F^i = t_{E}^i \circ
m_*^i$, for $i = 0,1$.

Let $j_{n_k, \infty}$ be the universal map from $C^*(E(n_k))$ into $\varinjlim
C^*(E(n_k))$. Let $\psi: F \to E(\infty)$ be the inverse of the isomorphism of Lemma
\ref{prp:IsomF}. Kribs and Solel define regular factor maps $m_{k,k+1}: E(n_{k+1}) \to
E(n_k)$ such that $E(\infty) = \varprojlim E(n_k)$. For each $k$, write $m_k:E(\infty)
\to E(n_k)$ for the induced factor map. In \cite[Theorem~6.3]{KribsSolel:JAMS07} Kribs
and Solel invoke \cite[Proposition~4.13]{KatsuraII} to show that there is an isomorphism
$\rho: \varinjlim (C^*(E(n_k)), j_{k, k+1}) \to C^*(E(\infty))$; it follows from the
arguments of \cite[Proposition~4.13]{KatsuraII}  that $\rho \circ j_{k,\infty} = \mu_{n_k}$.
Define $\pi:= \mu_\psi \circ \rho$. Since $\psi$ is an isomorphism, so is $\mu_\psi$, and
\begin{align*}
\pi(j_{n_k,\infty}(p_{n_k,\lambda}))
    &= \mu_\psi(\rho(j_{n_k, \infty}(p_{n_k,\lambda})))\\
    &= \mu_\psi \circ \mu_{m_k} (p_{n_k,\lambda})
     = \mu_\psi(t_{E(\infty)}^0 (\chi_{\psi(Z(\lambda, k))}))
     = t_F^0(\chi_{Z(\lambda, k)}).
\end{align*}
A similar calculation gives $\pi(j_{n_k,\infty}(s_{n_k,(e,\lambda)})) = t_F^1(\chi_{\{e\}
\times Z(\lambda,k)})$.
\end{proof}

\begin{proof}[Proof of Theorem~\ref{thm:CKUT}]
The identifications $C^*(E,n_k) = C^*(E(n_k))$ induce an isomorphism $\alpha:
C^*(E,\omega) \cong \varinjlim (C^*(E(n_k)), j_{k,k+1})$ such that
$\alpha(\rho_{(\mu,k)})) = j_{k,\infty} (p_\mu).$ So, if
\[
\pi: \varinjlim (C^*(E(n_k), j_{k,k+1}) \to C^*(F)
\]
is the isomorphism from the proof of Proposition~\ref{prp:IsomF},
we have $\pi \circ \alpha(\rho_{(\mu,k)}) = t_F^0(\chi_{Z(\mu,k)})$ for all $k \in \NN, \mu
\in E^{<n_k}$. Hence $\phi_{s,p,\psi} \circ \alpha^{-1} \circ \pi^{-1}$ is a homomorphism
of $C^*(F)$ that carries $t_F^0(\chi_{Z(\mu,k)})$ to $\psi_{(\mu,k)}$. Kribs and Solel show
that $E(\infty)$ has no cycles, so Lemma \ref{lem:EisoF} shows that $F$ has no cycles. So
\cite[Theorem~5.12]{KatsuraI} implies that $\psi_{s,p,\psi} \circ \alpha^{-1} \circ
\pi^{-1}$ is injective if and only if each $\psi_{(\mu,k)} \not = 0$. Since $\alpha$ and
$\pi$ are isomorphisms, the result follows.
\end{proof}

We now turn our attention to simplicity of $C^*(E, \omega)$. In
\cite[Section~9]{KribsSolel:JAMS07}, Kribs and Solel provide a necessary and sufficient
condition on the topological graph $E(\infty)$ for
$\varinjlim C^*(E, \omega)$ to be simple. In this section, we consider finite strongly
connected graphs, and we employ Perron--Frobenius theory, as well as Katsura's
characterisation of simplicity for $C^*$-algebras of topological graphs
\cite[Theorem~8.12]{KatsuraIII}, to improve upon Kribs and Solel's result to obtain a
necessary and sufficient condition in terms of $E$ and $\omega$ for simplicity of
$C^*(E, \omega)$ provided that the
terms $n_k$ in $\omega$ diverge to infinity. (If the $n_k$ are bounded then they are
eventually constant, and then $C^*(E, \omega) \cong C^*(E(N))$; and so simplicity of
$C^*(E, \omega)$ is characterised by \cite[Proposition~5.1]{BatesPaskEtAl:NYJM00}.)

The following technical result will be useful again later in our analysis of the
structure of the factor KMS states on $\Tt C^*(E, \omega)$.

\begin{lem}\label{lem:direct sum}
Let $E$ be a strongly connected finite directed graph with no sources, and take a
sequence $\omega = (n_k)^\infty_{k=1}$ of nonzero positive integers such that $n_k \mid
n_{k+1}$ for all $k$. Fix $k$ such that $\gcd(\Pp_E, n_k) = \gcd(\Pp_E, \omega)$. For
each equivalence class $\Lambda \in E^0/{\sim_{n_k}}$, let
\[
Q_{k,\Lambda} := \sum_{\mu \in E^{<n_k}, s(\mu) \in \Lambda} \pi_{(\mu,k)} \in \Tt(E, \omega).
\]
Then the $Q_{k,\Lambda}$ are nonzero mutually orthogonal projections, and
\[
\Tt(E, \omega) = \bigoplus_{\Lambda \in E^0/{\sim_{n_k}}} Q_{k,\Lambda} \Tt(E, \omega) Q_{k,\Lambda}.
\]
The images $P_{k,\Lambda}$ of the $Q_{k,\Lambda}$ in the quotient $C^*(E, \omega)$ are
also nonzero, and the direct summands $P_{k,\Lambda} C^*(E, \omega) P_{k,\Lambda}$ are
simple.
\end{lem}
\begin{proof}
For $\Lambda \in E^0/{\sim_{n_k}}$, we put
\[
\Theta_{k,\Lambda} := \sum_{\mu \in E^{<n_k}, s(\mu) \in \Lambda} \theta_{n_k,\mu} \in \Tt(E, n_k).
\]
The $\Theta_{k,\Lambda}$ are mutually orthogonal by Proposition~\ref{prp:components}, and
nonzero because the generators of $\Tt(E, n_k) \cong \Tt C^*(E(n_k))$ are all nonzero.

We claim that for $\alpha \in E^{<n_k}$ and $\mu, \nu \in E^* r(\alpha)$, we have
$\sum_{\Lambda} Q_{\Lambda} t_\mu \theta_{(\alpha,k)} t^*_\nu Q_\Lambda = t_\mu
\theta_{(\alpha,l)} t^*_\nu$. Let $\Lambda$ be the equivalence class of $\alpha$ under
$\approx_{E(n_k)}$. Since $C_{n_k}(s(r_{n_k}(\mu,\alpha)), s(\alpha)) =
C_{n_k}(s([\mu\alpha]_{n_k}), s(\alpha)) = 0$, we have $s(r_{n_k}(\mu,\alpha)) \in
\Lambda$. Similarly, $s(r_{n_k}(\nu,\alpha)) \in \Lambda$. Let $(t, q)$ be the universal
Toeplitz--Cuntz--Krieger $E(n_k)$-family in $\Tt C^*(E(n_k))$. We have
\begin{align*}
\Theta_{k,\Lambda} t_{n_k, \mu} \theta_{n_k, \alpha} t^*_{n_k,\nu} \Theta_{k,\Lambda}
    &= \sum_{\eta,\zeta \in E^{<n_k}, s(\eta),s(\zeta) \in \Lambda} (q_\eta t_{(\mu,\alpha)} t^*_{(\nu,\alpha)} q_{\zeta}) \\
    &= t_{(\mu,\alpha)} t^*_{(\nu,\alpha)}
    = t_{n_k, \mu} \theta_{n_k, \alpha} t^*_{n_k,\nu},
\end{align*}
and since the $\Theta_{k,\Lambda}$ are mutually orthogonal, the claim follows.

We now show that each  $i_{n_k, n_{k+1}}(\Theta_{k,\Lambda}) = \Theta_{k+1,\Lambda}$. We
calculate:
\begin{align*}
i_{n_k, n_{k+1}}(\Theta_{k,\Lambda})
    &= i_{n_k, n_{k+1}}\Big(\sum_{\eta \in E^{<n_k},s(\eta) \in \Lambda} \theta_{n_k,\eta}\Big)\\
    &= \sum_{\zeta \in E^{<n_{k+1}}, s([\zeta]_{n_k}) \in \Lambda} \theta_{n_{k+1},\zeta}
    = \sum_{\zeta \in E^{<n_{k+1}}, s(\zeta) \in \Lambda} \theta_{n_{k+1},\zeta}
    = \Theta_{k+1,\Lambda}.
\end{align*}
The preceding two paragraphs show that every element of the spanning family for $\Tt(E,
\omega)$ described in the final statement of Lemma~\ref{lem:omega mpctn} belongs to
$Q_{k,\Lambda} \Tt(E, \omega) Q_{k,\Lambda}$ for some $\Lambda$, giving the desired
direct-sum decomposition.

To see that the images $P_{k,\Lambda}$ of the $Q_{k,\Lambda}$ in $C^*(E,\omega)$ are
nonzero, observe that for each $\Lambda$, and any $v \in \Lambda$, we have $P_{k,\Lambda}
\ge \rho_{(v,k)} = p_{n_k, v}$, which is nonzero since all the generators of
$C^*(E(n_k))$ are nonzero. For the assertion about simplicity, observe that the
isomorphism $C^*(E, \omega) \cong C^*(F)$ determined by Proposition~\ref{prp:IsomF}
carries each $P_{k,\Lambda}$ to $t_F^1(\chi_{X_\Lambda})$, where the $X_\Lambda$ are the
minimal invariant subsets of $F^0$ described in Lemma~\ref{lem:invariantsubset}. For each
$\Lambda$, let $F_\Lambda$ be the topological subgraph of $F$ given by $F_\Lambda^0 =
X_\Lambda$ and $F_\Lambda^1 = r_f^{-1}(X_\Lambda)$. Since $F^0_\Lambda$ and $F^1_\Lambda$
are clopen in $F^0$ and $F^1$, there are canonical inclusions $C(F_\Lambda^0)
\hookrightarrow C(F^0)$ and $C(F_\Lambda^1) \hookrightarrow C(F^1)$, and it is easy to
verify that the universal property of $C^*(F_\Lambda)$ applied to these inclusions gives
surjective homomorphisms $\iota_\Lambda : C^*(F_\Lambda) \to t_F^1(\chi_{X_\Lambda}) C^*(F)
t_F^1(\chi_{X_\Lambda})$. Lemma~\ref{lem:invariantsubset} shows that each $F_\Lambda$ is
invariant. By \cite[Lemma~9.1]{KribsSolel:JAMS07} $E(\infty)$ has no loops, so
Lemma~\ref{lem:EisoF} shows that $F$ also has no loops, and hence each $F_\Lambda$ has no
loops. Hence \cite[Theorem~8.12]{KatsuraIII} shows that each $C^*(F_\Lambda)$ is simple.
Hence each $P_{k,\Lambda} C^*(E, \omega) P_{k,\Lambda} \cong C^*(F_\Lambda)$ is simple.
\end{proof}

\begin{cor}\label{cor:simple}
Let $E$ be a strongly connected finite directed graph with no sources, and take a
sequence $\omega = (n_k)^\infty_{k=1}$ of nonzero positive integers such that $n_k \mid
n_{k+1}$ for all $k$. Suppose that $n_k \to \infty$ as $k \to \infty$. Then $C^*(E,
\omega)$ is simple if and only if $\gcd(\Pp_E, \omega) = 1$.
\end{cor}
\begin{proof}
This follows immediately from the final statement of Lemma~\ref{lem:direct sum}.
\end{proof}

\section{KMS states}\label{sec:KMS}

In this section we study the KMS states for the gauge action on $\Tt(E, \omega)$.
Throughout this section, if $X$ is a compact topological space, then $\Mm^+_1(X)$ denotes
the Choquet simplex of Borel probability measures on $X$. We write $A_E$ for the
adjacency matrix $A_E(v,w)= |vE^1 w|$ of a finite graph $E$, and $\rho(A_E)$ for its
spectral radius.

The following summarises our main results about KMS states on $\Tt(E, \omega)$ and
$C^*(E, \omega)$.

\begin{thm}\label{thm:mainKMS}
Let $E$ be a finite strongly connected graph with no sources, and take a sequence $\omega
= (n_k)^\infty_{k=1}$ of nonzero positive integers such that $n_k \mid n_{k+1}$ for all
$k$. Let $\alpha : \RR \to \Aut\Tt(E, \omega)$ be given by $\alpha_t = \gamma_{e^{it}}$.
\begin{enumerate}
\item\label{it:main1} For $\beta > \ln\rho(A_E)$ there is an affine isomorphism
    (described in Corollary~\ref{cor:normalisedaffine}) of $\Mm_1^+(\varprojlim
    E^{<n_k})$ onto the KMS$_\beta$-simplex of $\Tt(E, \omega)$.
\item\label{it:main2} There are exactly $\gcd(\Pp_E, \omega)$ extremal
    KMS$_{\ln\rho(A_E)}$-states of $\Tt(E, \omega)$ (described explicitly in
    Theorem~\ref{thm:betterKMS}).
\item\label{it:main3} For $\beta < \ln\rho(A_E)$, there are no KMS$_\beta$ states for
    $\Tt(E, \omega)$.
\item\label{it:main4} A KMS$_\beta$ state of $\Tt(E, \omega)$ factors through $C^*(E,
    \omega)$ if and only if $\beta = \ln\rho(A_E)$.
\end{enumerate}
\end{thm}

\subsection{A transformation on finite signed Borel measures} Let $E$ be a finite
directed graph with no sources, and $\omega = (n_k)^\infty_{k=1}$ a sequence of positive
integers such that $n_k \mid n_{k+1}$ for all $k$. We consider the Banach space
$\Mm(\varprojlim E^{<n_k})$ of finite signed measures on the spectrum $\varprojlim
E^{<n_k}$ of the commutative subalgebra of $C^*(E, \omega)$ described in
Section~\ref{sec:universal}. We show that the vertex adjacency matrices $A_{E(n_k)}$
induce a bounded linear transformation $A_\omega$ of $\Mm(\varprojlim E^{<n_k})$. We use
Perron--Frobenius theory to show that $\|A_\omega\| = \rho(A_E)$, and that it always
admits a positive eigenmeasure. We provide a condition under which this eigenmeasure is
unique up to scalar multiples.

For $k \ge 1$, define a map $p^*_{n_{k+1}, n_k} : \Mm(E^{<n_{k+1}}) \to \Mm(E^{<n_k})$ by
$p^*_{n_{k+1}, n_k}(m)(U) = m(p_{n_{k+1}, n_k}^{-1}(U))$, where $U$ is a Borel measurable subset
of $E^{<n_k}$. Then $p^*_{n_{k+1}, n_k}$ is linear and the $(\Mm(E^{<n_k}), p^*_{n_{k+1}, n_k})$ form a
projective sequence of Banach spaces, giving a Fr\'echet space $\varprojlim (\Mm(E^{<n_k}),
p^*_{n_{k+1}, n_k})$ We can also form the Banach space $\Mm(\varprojlim E^{<n_k})$.
The following lemma describes an injective (but typically not surjective) linear map from the
latter into the former; the result must be standard, but it is also easy enough to give a quick proof.

\begin{lem}\label{lem:M-injection}
Let $E$ be a finite directed graph with no sources, and take a sequence $\omega =
(n_k)^\infty_{k=1}$ of nonzero positive integers such that $n_k \mid n_{k+1}$ for all
$k$. There is a continuous injective linear map $\iota_\omega : \Mm(\varprojlim E^{<n_k})
\to \varprojlim (\Mm(E^{<n_k}), p^*_{n_{k+1}, n_k})$ such that
$\iota_\omega(m)_k(\{\tau\}) = m(Z((\tau,k))$ for all $m,k,\tau$.
\end{lem}
\begin{proof}
For each $k \ge 1$, define $p^*_{\infty, n_k} : \Mm(\varprojlim E^{<n_k}) \to
\Mm(E^{<n_k})$ by $p^*_{\infty, n_k}(m)(\{\tau\}) = m(p^{-1}_{\infty, k}(\tau))$. Then
each $p^*_{\infty, n_k}$ is linear, and we have
\[
p^*_{n_{k+1},n_k}(p^*_{\infty, n_{k+1}}(m))(\{\tau\})
    = m(p^{-1}_{\infty, n_{k+1}}(p^{-1}_{n_{k+1}, n_k}(\tau)))
    = m(p^{-1}_{\infty, n_k}(\tau))
    = p^*_{\infty, n_k}(m)(\{\tau\})
\]
for all $k$. So the universal property of $\varprojlim (\Mm(E^{<n_k}), p^*_{n_{k+1},
n_k})$ implies that there is a continuous map $\iota_\omega$ such that
$\iota_\omega(m)_k(\tau) = m(Z(\tau,k))$ for all $m,k,\tau$. Direct calculation shows
that $\iota_\omega$ is linear.

For injectivity, take $m \in \Mm(\varprojlim E^{<n_k})$ with $\iota_\omega(m) = 0$. For
each $k \in \NN$ and $\mu \in E^{<n_k}$, we have $m(Z(\mu,k)) = \iota_\omega(m)_k(\{\mu\}) = 0$,
and since the $Z(\mu,k)$ are a basis for $\varprojlim E^{<n_k}$, we deduce that $m = 0$.
\end{proof}

\begin{rmk}
The map $\iota_\omega$ is typically not surjective.  For example, let $E$ be the directed
graph with one vertex $v$ and one edge $e$. Define $m_0 \in \Mm(E^0)$ by $m_0(\{v\}) =
1$. Let $n_k = 2^k$ for all $k$, and inductively define $m_k \in \Mm(E^{<n_k})$ by
\[
m_k(\{e^j\}) = 2m_{k-1}(\{e^j\})\quad\text{ and } m_k\big(\{e^{j+2^{k-1}}\}\big) = -m_{k-1}(\{e^j\})
\]
for $j \in \{0, \dots, 2^{k-1} -1\}$. Then $(m_k)^\infty_{k=1} \in \varprojlim
\Mm(E^{<n_k})$, but we have $m_k(\{v\}) = 2^k \to \infty$. For any $m \in \Mm(\varprojlim
E^{<n_k})$, we have $\iota_\omega(m)_k(\{v\}) = m(Z(k,\tau)) \le m^+(Z(k,\tau))$ for all
$k$, so the sequence $\iota_\omega(m)_k(\{v\})$ is bounded. So $(m_k)^\infty_{k=1}$ does
not belong to the range of $\iota_\omega$.
\end{rmk}

In what follows, if $m \in \Mm(\varprojlim E^{<n_k})$, we will frequently write $m_{n_k}$
for $\iota_\omega(m)_k \in \Mm(E^{<n_k})$. We also regard the adjacency matrix
$A_{E(n_k)}$ as a linear transformation of the finite-dimensional vector space
$\Mm(E^{<n_k}) \cong \RR^{E^{<n_k}}$. We show how the $A_{E(n_k)}$ induce a linear
transformation of $\varprojlim \Mm(E^{<n_k})$.

\begin{lem}\label{lem:Ank system}
Let $E$ be a finite directed graph with no sources, and take a sequence $\omega =
(n_k)^\infty_{k=1}$ of nonzero positive integers such that $n_k \mid n_{k+1}$ for all
$k$. For $k \in \NN$ let $A_{n_k} := A_{E(n_k)}$, regarded as a linear transformation of
$\Mm(E^{<n_k})$. For $m \in \Mm(E^{<n_k})$, we have
\begin{equation}\label{eq:Ank action}
(A_{n_k}m)(\{\mu\})
    = \begin{cases}
        m(\{\mu_2\dots\mu_{|\mu|}\}) &\text{ if $\mu \in E^{<n_k} \setminus E^0$} \\
        \sum_{e \nu \in \mu E^{n_k}} m(\{\nu\}) &\text{ if $\mu \in E^0$,}
    \end{cases}
\end{equation}
and
\begin{equation}\label{eq:Ank compatible}
    A_{n_{k-1}}(p^*_{n_k, n_{k-1}}(m)) = p^*_{n_k, n_{k-1}}(A_{n_k}(m)).
\end{equation}
\end{lem}
\begin{proof}
We write $\{\delta_{\mu,k} : \mu \in E^{<n_k}\}$ for the basis of Dirac measures on
$E^{<n_k}$. We have
\begin{align*}
A_{n_k}(\delta_{\mu,k})
    &= \sum_{\nu \in E^{<n_k}} A_{n_k}(\nu,\mu)\delta_{\nu,k} \\
    &= \sum_{\nu \in E^{<n_k}} |\nu E(n_k)^1 \mu| \delta_{\nu,k}
    = \begin{cases}
        \delta_{\mu_2\dots\mu_{|\mu|}, k} &\text{ if $\mu \in E^{<n_k} \setminus E^0$} \\
        \sum_{e \nu \in \mu E^{n_k}} \delta_{\nu, k} &\text{ if $\mu \in E^0$.}
    \end{cases}
\end{align*}
Now~\eqref{eq:Ank action} follows from linearity.

To prove~\eqref{eq:Ank compatible}, first consider $\mu \in E^{<n_{k-1}} \setminus E^0$.
We have
\begin{align*}
A_{n_{k-1}}(p^*_{n_k, n_{k-1}}(m))(\{\mu\})
    &= p^*_{n_k, n_{k-1}}(m)(\{\mu_2\dots\mu_{|\mu|}\})
    = \sum_{\tau \in E^{<n_k}, [\tau]_{n_{k-1}} = \mu_2\dots\mu_{|\mu|}} m(\{\tau\}) \\
    &= \sum_{\eta \in E^{<n_k}, [\eta]_{n_{k-1}} = \mu} A_{n_k}(m)(\{\eta\})
    = p^*_{n_k, n_{k-1}}(A_{n_k}(m))(\{\mu\}).
\end{align*}
Now consider $\mu = v \in E^0$. We have
\begin{align*}
A_{n_{k-1}}(p^*_{n_k, n_{k-1}}(m))(\{v\})
    &= \sum_{e\tau \in vE^{n_{k-1}}} p^*_{n_k, n_{k-1}}(m)(\{\tau\})
    = \sum_{\substack{e \in vE^1, \lambda \in s(e)E^{<n_k}\\|e\lambda| \in n_{k-1}\NN}} m(\{\lambda\})\\
    &= \sum_{\lambda \in E^{<n_k}, [\lambda]_{n_{k-1}} = v} A_{n_k}(m)(\{\lambda\})
    = p^*_{n_k, n_{k-1}}(A_{n_k}(m))(\{v\}).\qedhere
\end{align*}
\end{proof}

\begin{prp}\label{prp:Aomega}
Let $E$ be a finite directed graph with no sources, and take a sequence $\omega =
(n_k)^\infty_{k=1}$ of nonzero positive integers such that $n_k \mid n_{k+1}$ for all
$k$. For $k \in \NN$ let $A_{n_k} := A_{E(n_k)}$, regarded as a linear transformation of
$\Mm(E^{<n_k})$. There is a linear transformation $A_\omega$ of $\varprojlim
\Mm(E^{<n_k})$ given by $A_\omega m = (A_{n_1} m_1, A_{n_2} m_2, \dots)$. The inclusion
$\iota_\omega$ of Lemma~\ref{lem:M-injection} satisfies
\[
A_\omega(\iota_\omega(\Mm(\varprojlim E^{<n_k}))) \subseteq \iota_\omega(\Mm(\varprojlim E^{<n_{k}}))
\]
\end{prp}
\begin{proof}
Fix $(m_1, m_2, \dots) \in \varprojlim \Mm(E^{<n_k})$. By Lemma~\ref{lem:Ank system} we
have $p^*_{n_k, n_{k-1}}(A_{n_k}(m_{n_k})) = A_{n_{k-1}}(p^*_{n_k, n_{k-1}})(m_{n_k}) =
A_{n_k-1}m_{n_{k-1}}$, so $(A_{n_1} m_1, A_{n_2} m_2, \dots) \in \varprojlim
\Mm(E^{<n_k})$. The universal property of $\varprojlim \Mm(E^{<n_k})$ gives a continuous
map $A_\omega : \varprojlim \Mm(E^{<n_k}) \to \varprojlim \Mm(E^{<n_k})$ satisfying
$A_\omega m = (A_{n_1} m_1, A_{n_2} m_2, \dots)$. It is clear that $A_\omega$ is linear.

By Lemma~\ref{lem:Ank system}, we have $p^*_{n_k, n_{k-1}}(A_{n_k}m_{n_k}^+) =
A_{n_{k-1}} m_{n_{k-1}}^+$. So by \cite[Theorem~2.2]{Choksi:PLMS58}, there is a positive
Borel measure $M^+$ on $\varprojlim E^{<n_k}$ such that $M^+(Z(\mu,k)) =
(A_{n_k}m_{n_k}^+)(\{\mu\})$ for all $k \in \NN$ and $\mu \in E^{<n_k}$. Similarly, there
is a positive Borel measure $M^-$ on $\varprojlim E^{<n_k}$ such that $M^-(Z(\mu,k)) =
(A_{n_k}m_{n_k}^-)(\{\mu\})$ for $\mu \in E^{<n_k}$. Now $A_\omega\iota_\omega(m) =
\iota_\omega(M^+ - M^-)$ belongs to the range of $\iota_\omega$.
\end{proof}

For calculations later, we will want to understand the transformation $A_\omega$ in terms
of the measures of cylinder sets.

\begin{lem}\label{lem:Aomega on cylinders}
Let $E$ be a finite directed graph with no sources, and take a sequence $\omega =
(n_k)^\infty_{k=1}$ of nonzero positive integers such that $n_k \mid n_{k+1}$ for all
$k$. For $m \in \Mm(\varprojlim E^{<n_k})$, $k \in \NN$ and $\mu \in E^{<n_k}$, the
transformation $A_\omega$ of Proposition~\ref{prp:Aomega} satisfies
\begin{align*}
(A_\omega m)(Z(\mu,k))
    &= \begin{cases}
        m(Z(\mu_2\dots\mu_{|\mu|}, k)) &\text{ if $\mu \in E^{<n_k}\setminus E^0$} \\
        \sum_{e\nu \in \mu E^{n_k}} m(Z(\nu,k)) &\text{ if $\mu \in E^0$}
    \end{cases}\\
    &= \sum_{\nu \in E^{<n_k}} |\mu E(n_k)^1\nu| m(Z(\mu,k)).
\end{align*}
\end{lem}
\begin{proof}
Since $A_\omega m(Z(\mu,k)) = A_\omega m(p^{-1}_{\infty,n_k}(\{\mu\})) =
A_{n_k}m_{n_k}(\{\mu\})$, the result follows from Lemma~\ref{lem:Ank system}.
\end{proof}

We now show that $A_\omega$ admits a positive eigenmeasure and also that the norm of
$A_\omega$, as an operator on the Banach space $\Mm(\varprojlim E^{<n_k})$, is
$\rho(A_E)$. Recall that the unimodular Perron-Frobenius eigenvector of an irreducible
nonnegative matrix $A$ is its unique positive eigenvector with unit 1-norm.

\begin{prp}\label{prp:Aomega eigenmeasure}
Let $E$ be a finite strongly connected directed graph with no sources, and take a
sequence $\omega = (n_k)^\infty_{k=1}$ of nonzero positive integers such that $n_k \mid
n_{k+1}$ for all $k$. Let $x^E$ be the unimodular Perron-Frobenius eigenvector of $A_E$.
The transformation $A_\omega$ of Proposition~\ref{prp:Aomega} admits a positive
eigenmeasure $m$ such that
\begin{equation}\label{eq:PF measure}
m(Z(\mu,k)) = \frac{1}{n_k} \rho(A_E)^{-|\mu|} x^E_{s(\mu)}\quad\text{ for all $\mu \in E^{<n_k}$,}
\end{equation}
and the corresponding eigenvalue is $\rho(A_E)$, and is equal to the operator norm of
$A_\omega$ as a transformation of $\Mm(\varprojlim E^{<n_k})$.
\end{prp}
\begin{proof}
To see that~\eqref{eq:PF measure} specifies an element $m \in \Mm(\varprojlim E^{<n_k})$,
define measures $m_k$ by $m_k(\{\mu\}) := \frac{1}{n_k}\rho(A_E)^{-|\mu|} x^E_{s(\mu)}$
for $\mu \in E^{<n_k}$. Let $a_k := n_{k+1}/n_k$ for each $k$. Using at the fifth
equality that $A^j_E x^E = \rho(A_E)^j x_E$ for all $j$, we calculate
\begin{align*}
p^*_{n_{k+1}, n_k}(m_{n_{k+1}})(\{\mu\})
    &= \sum_{\tau \in E^{<n_{k+1}}, [\tau]_{n_k} = \mu} \frac{1}{n_{k+1}} \rho(A_E)^{-|\tau|} x^E_{s(\tau)} \\
    &= \sum^{a_k - 1}_{j=0} \frac{1}{n_k}\rho(A_E)^{|\mu|} \sum_{\lambda \in s(\mu)E^{jn_k}} \frac{1}{a_k} \rho(A_E)^{-jn_k} x^E_{s(\lambda)} \\
    &= \sum^{a_k - 1}_{j=0} \frac{1}{n_k}\rho(A_E)^{|\mu|} \frac{1}{a_k} \rho(A_E)^{-jn_k} \sum_{w \in E^0} A_E^{jn_k}(s(\mu),w) x^E_w \\
    &= \sum^{a_k - 1}_{j=0} \frac{1}{n_k}\rho(A_E)^{|\mu|} \frac{1}{a_k} \rho(A_E)^{-jn_k} (A_E^{jn_k} x^E)_{s(\mu)} \\
    &= \sum^{a_k - 1}_{j=0} \frac{1}{n_k}\rho(A_E)^{|\mu|} \frac{1}{a_k} x^E_{s(\mu)}
    = \frac{1}{n_k}\rho(A_E)^{-|\mu|} x^E_{s(\mu)} = m_{n_k}(\{\mu\}).
\end{align*}
Now \cite[Theorem~2.2]{Choksi:PLMS58} implies that there is a positive measure $m$ on
$\varprojlim E^{<n_k}$ satisfying~\eqref{eq:PF measure}.

To see that $m$ is an eigenmeasure for $A_\omega$ with eigenvalue $\rho(A_E)$, observe
that for $\mu \in E^{<n_k} \setminus E^0$, we have
\[
(A_\omega m)(Z(\mu, k))
    = m(Z(\mu_2\dots\mu_{|\mu|}, k))
    = \frac{1}{n_k} \rho(A_E)^{-|\mu|+1} x^E_{s(\mu)}
    =\rho(A_E) m(Z(\mu,k)),
\]
and for $v \in E^0$, we have
\begin{align*}
(A_\omega m)(Z(v,k))
    &= \sum_{e \in vE^1, \tau \in s(e)E^{n_k-1}} \frac{1}{n_k}\rho(A_E)^{-|\tau|} x^E_{s(\tau)}
    = \frac{1}{n_k}\sum_{w \in E^0} \sum_{\lambda \in vE^{n_k} w} \rho(A_E)^{-|\lambda|+1} x^E_w \\
    &= \frac{1}{n_k}(A^{n_k}\rho(A_E)^{-n_k+1} x^E_w)_v
    = \frac{1}{n_k}\rho(A_E)x^E_v.
\end{align*}
So $m$ is an eigenmeasure for $A_\omega$ with corresponding eigenvalue $\rho(A_E)$. It
follows immediately that $\|A_\omega\| \ge \rho(A_E)$. For the reverse inequality, take
$m \in \Mm(\varprojlim E^{<n_k})$ and consider its Jordan decomposition $m = m^+ - m^-$.
Since $A_\omega$ is linear, we have $A_\omega m^+ - A_\omega m^- = A_\omega m$, and since
the $A_{n_k}$ are positive matrices, the measures $A_\omega m^{\pm}$ are positive
measures. So the Jordan Decomposition Theorem implies that $A_\omega m^+ \ge (A_\omega
m)^+$ and $A_\omega m^- \ge (A_\omega m)^-$. So
\begin{align*}
\|A_\omega\|
    &= \sup_{\|m\| = 1} \|A_\omega m\|
     = \sup_{\|m\| = 1} \big((A_\omega m)^+(\varprojlim E^{< n_k}) + (A_\omega m)^-(\varprojlim E^{< n_k})\big)\\
    &\le \sup_{\|m\| = 1} \big((A_\omega m^+)(\varprojlim E^{< n_k}) + (A_\omega m^-)(\varprojlim E^{< n_k})\big)\\
    &= \sup_{\|m\| = 1} \big((A_1 m_1^+)(E^0) + (A_1 m_1^-)(E^0)\big)
     \le \sup_{\|m\| = 1} \big(\rho(A_E) m_1^+(E^0) + \rho(A_E)m_1^-(E^0)\big)\\
    &= \rho(A_E) \sup_{\|m\| = 1}\big(m^+(\varprojlim E^{< n_k}) +  m^-(\varprojlim E^{< n_k})\big)
     = \rho(A_E).\qedhere
\end{align*}
\end{proof}

We now show that if $E$ is strongly connected and $\gcd(\Pp_E, \omega) = 1$, then the
measure $m$ of the preceding proposition is the only positive probability measure that is
an eigenmeasure for the transformation $A_\omega$.

\begin{lem}\label{lem:uniquePF}
Let $E$ be a finite strongly connected directed graph with no sources, and take a
sequence $\omega = (n_k)^\infty_{k=1}$ of nonzero positive integers such that $n_k
\mid n_{k+1}$ for all $k$. Let $m$ be the measure of Proposition~\ref{prp:Aomega
eigenmeasure}, and fix $k$ such that $\gcd(\Pp_E, n_k)  = \gcd(\Pp_E, \omega)$.
\begin{enumerate}
\item\label{it:mLambdas} Let $\sim_{n_k}$ be the equivalence relation on $E^0$ of
    Lemma~\ref{lem:Cn}. For $\Lambda \in E^0/{\sim_{n_k}}$, let $X_\Lambda =
    \bigcup_{\mu \in E^{<n_k}, s(\mu) \in \Lambda} Z(\mu, k) \subseteq
    \varprojlim E^{<n_k}$, and define $m^\Lambda \in \Mm(\varprojlim E^{<n_k})$
    by $m^\Lambda(U) := \frac{1}{m(X_\Lambda)} m(U \cap X_\Lambda)$. Then each
    $m^\Lambda$ is a normalised eigenmeasure for $A_\omega$ with eigenvalue
    $\rho(A_E)$.
\item\label{it:radii} For each $l \ge k$, and for each $\Lambda \in
    E^0/{\sim_{n_k}}$, the block $A^\Lambda_{n_l} \in M_{E^{<n_l}\Lambda}(\ZZ)$ of
    $A_{n_l}$ is an irreducible matrix. We have $\rho(A^\Lambda_{n_l}) = \rho(A_E)$
    and $m^\Lambda_{n_l} = (m^\Lambda(Z(\mu, l)))_{\mu \in E^{<n_l}}$ is the
    unimodular Perron--Frobenius eigenvector of $A^\Lambda_{n_l}$.
\item\label{it:unique} Every positive eigenmeasure for $A_\omega$ is a convex
    combination of the $m^\Lambda$.
\end{enumerate}
\end{lem}
\begin{proof}
(\ref{it:mLambdas}) Proposition~\ref{prp:Aomega eigenmeasure} shows that $m$ is an
eigenmeasure with $A_\omega m = \rho(A_E) m$. Lemma~\ref{lem:invariantsubset} shows that
each $\Mm(X_\Lambda) \subseteq \Mm(\varprojlim E^{<n_k})$ is invariant for $A_\omega$,
and it follows that $A_\omega m^\Lambda = \rho(A_E) m^\Lambda$ for each $\Lambda$.

(\ref{it:radii}) For each $l \ge k$ and each $\Lambda \in E^0 / \sim_{n_l}$, the matrix
$A^\Lambda_{n_l}$ is irreducible by Proposition~\ref{prp:components}. By definition of
$A_\omega$, we have $A_\omega \chi_{Z(\mu,l)} = \sum_\nu A^\Lambda_{n_l}(\nu,\mu)
\chi_{Z(\nu,l)}$, and so (\ref{it:mLambdas}) shows that $A^\Lambda_{n_l}
m^{\Lambda}_{n_l} = \rho(A_E)m^\Lambda_{n_l}$. The Perron-Frobenius theorem
\cite[Theorem~1.5]{Seneta:NNMMC} implies that every entry of the Perron--Frobenius
eigenvector of the irreducible matrix $A_E$ is nonzero, and so \eqref{eq:PF
measure} shows that $m^\Lambda_{n_l}$ is the unimodular Perron--Frobenius eigenvector of
$A^\Lambda_{n_l}$, and so its eigenvalue $\rho(A_E)$ is equal to
$\rho(A^\Lambda_{n_l})$.

(\ref{it:unique}) Suppose that $m' \in \Mm^+(\varprojlim E^{<n_k})$ and $z \in \CC$
satisfy $A_\omega m' = zm'$. Then in particular $A_{n_l} (m')^\Lambda_{n_l} =
(zm')^\Lambda_{n_l}$ for each $l \ge k$ and $\Lambda \in E^0/{\sim_{n_l}}$. Since each
$A^\Lambda_{n_l}$ is irreducible, this forces $z = \rho(A_{n_l}) = \rho(A_E)$, and
$(m')^\Lambda_{n_k}$ is a scalar multiple of $m^\Lambda_{n_l}$, so $m' = \sum_\Lambda
t_\Lambda m^\Lambda_{n_l}$. Since the supports of the $m^\Lambda_{n_l}$ are disjoint and
$m'$ is positive, the $t_\Lambda$ are positive, and their sum is $1$ because $m'$ and the
$m^\Lambda_{n_l}$ are normalised. Since this is true for all $l$, continuity implies
that $m' = \sum_\Lambda t_\Lambda m^\Lambda$.
\end{proof}

\begin{lem}\label{lem:subinvariance}
Let $E$ be a finite strongly connected directed graph with no sources, and take a
sequence $\omega = (n_k)^\infty_{k=1}$ of nonzero positive integers such that $n_k \mid
n_{k+1}$ for all $k$. Suppose that $s > 0$ and $m \in \Mm^+(\varprojlim E^{<n_k})$
satisfy $A_\omega m \le s m$. Then $s \ge \rho(A_E)$. Moreover, $s = \rho(A_E)$ if and
only if $A_\omega m = sm$.
\end{lem}
\begin{proof}
Since $A_\omega m \le sm$, we have $A_E m_1 \le s m_1$, and since $A_E$ is irreducible,
the subinvariance theorem \cite[Theorem~1.6]{Seneta:NNMMC} implies that $s \ge
\rho(A_E)$.

Suppose that $s = \rho(A_E)$. For $k$ such that $\gcd(\Pp_E, n_k) = \gcd(\Pp_E, \omega)$,
the matrix $A^\Lambda_{n_k}$ is irreducible by Proposition~\ref{prp:components}, so the
forward implication of the last assertion of \cite[Theorem~1.6]{Seneta:NNMMC} implies
that $A^\Lambda_{n_k} m_{n_k} = \rho(A^\Lambda_{n_k}) m_{n_k}$. Since
$\rho(A^\Lambda_{n_k}) = \rho(A_E)$ for all $k$ by part~(\ref{it:radii}) of
Lemma~\ref{lem:uniquePF}, we deduce that $A_{n_k} m_{n_k} = \rho(A_E) m_{n_k}$ for all
$k$. So $A_\omega m = \rho(A_E)m$.

Now suppose that $A_\omega m = sm$. Then part~(\ref{it:unique}) of
Lemma~\ref{lem:uniquePF} gives $s = \rho(A_E)$.
\end{proof}

\subsection{Characterising KMS states} We characterise the KMS$_\beta$-states for the
gauge action on $\Tt(E, \omega)$ in terms of their values at spanning elements $t_\mu
\pi_{(\alpha, k)} t^*_\nu$. We describe a subinvariance condition on the measure $m^\phi$
on $\varprojlim E^{<n_k}$ induced by a KMS state $\phi$. We also show that a KMS state
factors through $C^*(E, \omega)$ if and only if this subinvariance condition is
invariance. Our approach follows the general program of \cite{LacaRaeburn:AM10}, but is
by now quite streamlined.

\begin{thm}\label{thm:KMSchar}
Let $E$ be a finite directed graph with no sources, and take a sequence $\omega =
(n_k)^\infty_{k=1}$ of nonzero positive integers such that $n_k \mid n_{k+1}$ for all
$k$. Let $\alpha : \RR \to \Aut \Tt(E, \omega)$ be given by $\alpha_t = \gamma_{e^{it}}$.
Let $\beta \in \RR$.
\begin{enumerate}
\item\label{it:KMS cond} A state $\phi$ of $\Tt(E,\omega)$ is a KMS$_\beta$ state for
    $\alpha$ if and only if
    \begin{equation}\label{eq:KMS on spanner}
        \phi(t_\mu \pi_{(\tau,k)} t^*_\nu) = \delta_{\mu,\nu} e^{-\beta|\mu|} \phi(\pi_{(\tau,k)})
    \end{equation}
    for all $k \in \NN$, all $\tau \in E^{<n_k}$ and all $\mu,\nu \in E^*r(\tau)$.
\item\label{it:KMS subinvariance} Suppose that $\phi$ is a KMS$_\beta$ state for
    $(\Tt(E, \omega), \alpha)$, and let $m^\phi$ be the measure on $\varprojlim
    E^{<n_k}$ such that $m^\phi(Z(\mu,k)) = \phi(\pi_{(\mu,k)})$ for $\mu \in
    E^{<n_k}$. Then $m^\phi$ is a probability measure and satisfies the subinvariance
    relation $A_\omega m^\phi \le e^\beta m^\phi$.
\item\label{it:KMS factor} A KMS$_\beta$ state $\phi$ of $(\Tt(E, \omega), \alpha)$
    factors through $C^*(E, \omega)$ if and only if $A_\omega m^\phi = e^\beta
    m^\phi$.
\end{enumerate}
\end{thm}
\begin{proof}
(\ref{it:KMS cond}) Suppose that $\phi$ is KMS. Then $\phi$ is $\alpha$-invariant---by
\cite[Proposition~5.33]{BratteliRobinson:OAQSMvII} if $\beta \not= 0$, or by definition
if $\beta = 0$---and so also $\gamma$-invariant, and then
\[
\phi(t_\mu \pi_{(\tau, k)} t^*_\nu)
    = \int_\TT \phi(\gamma_z(t_\mu \pi_{(\tau, k)} t^*_\nu))\,dz
    = \int_\TT z^{|\mu|-|\nu|}\,dz  \phi(t_\mu \pi_{(\tau, k)} t^*_\nu),
\]
which is zero if $|\mu| \not= |\nu|$. If $|\mu| = |\nu|$, then the KMS condition gives
\[
\phi(t_\mu \pi_{(\tau, k)} t^*_\nu)
    = e^{-\beta|\mu|} \phi(t^*_\nu t_\mu \pi_{(\tau, k)})
    = \delta_{\mu,\nu} e^{-\beta |\mu|} \phi(\pi_{(\tau, k)}).
\]
Now suppose that $\phi$ satisfies~\eqref{eq:KMS on spanner}. Then the argument of
\cite[Proposition~2.1(a)]{anHuefLacaEtAl:JMAA2013} shows that $\phi$ is KMS.

(\ref{it:KMS subinvariance}) We have $m^\phi \ge 0$ because $\phi$ is a state. To see
that $m^\phi$ is a probability measure, just observe that $\phi$ restricts to a state of
$\pi(C_0(\varprojlim E^{<n}))$, and so $m^\phi$ is a probability measure by the Riesz
representation theorem. To see that it satisfies the subinvariance condition, we
calculate:
\begin{align}
\sum_{e \in r(\mu)E^1} \phi(t_e t^*_e \pi_{(\mu,k)})
    &= \sum_{e \in r(\mu)E^1}  e^{-\beta}\phi(t^*_e \pi_{(\mu,k)} t_e) \nonumber\\
    &= e^{-\beta}\begin{cases}
            \phi(\pi_{(\mu_2\dots\mu_{|\mu|}, k)} t^*_{\mu_1} t_{\mu_1}) &\text{ if $\mu \not\in E^0$} \\
            \sum_{e\nu \in r(\nu) E^{n_k}} \phi(\pi_{(\nu, k)} t^*_e t_e) &\text{ if $\mu \in E^0$}
        \end{cases}\nonumber\\
    &= e^{-\beta}\begin{cases}
            m^\phi(Z(\mu_2\dots\mu_{|\mu|}, k)) &\text{ if $\mu \not\in E^0$} \\
            \sum_{e\nu \in r(\nu) E^{n_k}} m^\phi(Z(\nu,k)) &\text{ if $\mu \in E^0$}
        \end{cases}\nonumber\\
    &= e^{-\beta} A_\omega m^\phi(Z(\mu,k))\label{eq:halfway}
\end{align}
by Lemma~\ref{lem:Aomega on cylinders}. Hence each
\[
e^\beta m^\phi(Z(\mu,k))
    = e^\beta \phi(\pi_{(\mu,k)})
    = e^\beta \phi(p_{r(\mu)}\pi_{(\mu,k)})
    \ge \sum_{e \in r(\mu)E^1} e^\beta \phi(t_e t^*_e \pi_{(\mu,k)})
    = A_\omega m^\phi(Z(\mu,k)).
\]

(\ref{it:KMS factor}) Recall that $C^*(E, \omega)$ is the quotient of $\Tt(E, \omega)$ by
the ideal generated by the projections $q_v - \sum_{e \in vE^1} t_e t^*_e$, $v \in E^0$.
Thus by Lemma~2.2 of \cite{anHuefLacaEtAl:JMAA2013} it suffices to check that
$\phi\big(q_v - \sum_{e \in vE^1} t_e t^*_e\big) = 0$ for all $v$ if and only if
$A_\omega m^\phi = e^\beta m^\phi$. For each $v \in E^0$ and $k \ge 1$, we have
\[
q_v - \sum_{e \in vE^1} t_e t^*_e
    =\sum_{\mu \in vE^{<n_k}} \Big(q_{r(\mu)} - \sum_{e \in r(\mu)E^1} t_e t^*_e\Big)\pi_{(\mu,k)}.
\]
Since each term in the last sum is nonnegative, $\phi\big(q_v - \sum_{e \in vE^1} t_e
t^*_e\big) = 0$ for each $v$ if and only if $\phi\big(\big(q_{r(\mu)} - \sum_{e \in
r(\mu)E^1} t_e t^*_e\big)\pi_{(\mu,k)}\big) = 0$ for all $\mu \in E^{<n_k}$.
By~\eqref{eq:halfway} we have
\begin{align*}
\phi\Big(\Big(q_{r(\mu)} - \sum_{e \in r(\mu)E^1} t_e t^*_e\Big)\pi_{(\mu,k)}\Big)
    &= \phi\Big(\pi_{(\mu,k)} - \sum_{e \in r(\mu)E^1} t_e t^*_e\pi_{(\mu,k)}\Big) \\
    &= e^\beta m^\phi(Z(\mu,k)) - (A_\omega m^\phi)(Z(\mu,k)),
\end{align*}
and the result follows.
\end{proof}

\subsection{Constructing KMS states at large inverse temperatures}

In this section, for each measure $m$ satisfying the subinvariance relation of
Theorem~\ref{thm:KMSchar}(\ref{it:KMS subinvariance}) we construct a KMS state of $\Tt(E,
\omega)$ that induces $m$. We also show that positive subinvariant measures $m$ are in
bijection with positive Borel probability measures on $\varprojlim E^{<n_k}$. Let
\[
E^* \times_{E^0} \varprojlim E^{<n_k} := \{(\lambda, x) : \lambda \in E^*, x \in \varprojlim E^{<n_k}, s(\lambda) = r(x_1)\}.
\]
Let $\{h_{\lambda, x} : (\lambda, x) \in E^* \times_{E^0} \varprojlim E^{<n_k}\}$ be the
canonical basis for $\ell^2(E^* \times_{E^0} \varprojlim E^{<n_k})$.

It is not hard to check using a sequential argument that $x \mapsto (r_{n_i}(\lambda,
x_i))^\infty_{i=1}$ is continuous from $\varprojlim E^{<n_k}$ to $\varprojlim E^{<n_k}$.
So for a finite graph $E$ and each $\lambda \in E^*$, there is a map $\alpha_\lambda :
C(\varprojlim E^{<n_k}) \to C(\varprojlim E^{<n_k})$ such that
\[
\alpha_\lambda(\chi_{Z(\mu, k)})(x) :=
    \begin{cases}
        \chi_{Z(\mu, k)}((r_{n_i}(\lambda, x_i))^\infty_{i=1}) &\text{ if $s(\lambda) = r(x)$}\\
        0   &\text{ otherwise.}
    \end{cases}
\]

\begin{prp}\label{prp:path rep}
Let $E$ be a row-finite directed graph with no sources, and take a sequence $\omega =
(n_k)^\infty_{k=1}$ of nonzero positive integers such that $n_k \mid n_{k+1}$ for all
$k$. There is a representation $\varsigma : \Tt(E, \omega) \to \Bb(\ell^2(E^*
\times_{E^0} \varprojlim E^{<n_k}))$ such that for $e \in E^1$ and $v \in E^0$,
\[
\varsigma(t_e)h_{\lambda, x} = \delta_{r(\lambda), s(e)} h_{e\lambda, x}
\qquad\text{ and }\qquad
\varsigma(q_v)h_{\lambda, x} = \delta_{r(\lambda), v} h_{\lambda, x},
\]
and such that for $\mu \in E^{<n_k}$, we have $\varsigma(\pi_{(\mu,k)}) h_{\lambda, x} =
\alpha_\lambda(\chi_{Z(\mu,k)})(x)h_{\lambda, x}$
\end{prp}
\begin{proof}
We aim to invoke the universal property of $\Tt(E, \omega)$. It is routine to check that
the formulas given for $\varsigma(t_e)$ and $\varsigma(q_v)$ define a
Toeplitz--Cuntz--Krieger $E$-family $(T, Q)$ in $\Bb(\ell^2(E^* \times_{E^0} \varprojlim
E^{<n_k}))$.

Likewise, for each $k$, the formula given for the $\varsigma(\pi_{(\mu,k)})$ determines
mutually orthogonal projections indexed by $\mu \in E^{<n_k}$ and satisfying
$\varsigma(\pi_{(\mu, k)}) = \sum_{\nu \in E^{<n_{k+1}}, [\nu]_{n_k} = \mu}
\varsigma(\pi_{(\nu, k+1)})$, so they determine a homomorphism $\tilde\varsigma :
C(\varprojlim E^{<n_k}) \to \Bb(\ell^2(E^* \times_{E^0} \varprojlim E^{<n_k}))$.

We show that $(T, Q, \tilde\varsigma)$ is a Toeplitz $\omega$-representation of $E$. Take
$e \in E^1$ and $\mu \in E^{<n_k}$ and suppose that $\mu = e\mu'$. For any $(\lambda, x)
\in E^* \times_{E^0} \varprojlim E^{<n_k}$, we have
\[
T^*_e \tilde\varsigma_{(\mu,k)} h_{\lambda ,x}
    = T^*_e \alpha_\lambda(\chi_{Z(\mu,k)})(x) h_{\lambda,x}
    = \begin{cases}
        \alpha_\lambda(\chi_{Z(\mu,k)})(x) h_{\lambda', x} &\text{ if $\lambda = e\lambda'$}\\
        0 &\text{ otherwise.}
    \end{cases}
\]
Also,
\[
\tilde\varsigma_{(\mu',k)} T^*_e h_{\lambda ,x}
    = \begin{cases}
        \varsigma_{(\mu', k)} h_{\lambda', x} &\text{ if $\lambda = e\lambda'$}\\
        0 &\text{ otherwise}
    \end{cases}
    = \begin{cases}
        \alpha_{\lambda'}(\chi_{Z(\mu',k)})(x) h_{\lambda', x} &\text{ if $\lambda = e\lambda'$}\\
        0 &\text{ otherwise.}
    \end{cases}
\]
If $\lambda \not= e\lambda'$ then both $T^*_e \tilde\varsigma_{(\mu,k)} h_{\lambda ,x}$
and $\tilde\varsigma_{(\mu',k)} T^*_e h_{\lambda ,x}$ are zero, so suppose that $\lambda
= e\lambda'$. Then $\alpha_\lambda(\chi_{Z(\mu,k)})(x) = \chi_{Z(\mu,k)}(r_{n_i}(\lambda,
x_i)^\infty_{i=1}) = 1$ if and only if $\alpha_{\lambda'}(\chi_{Z(\mu',k)})(x) = 1$ as
well; so $T^*_e \tilde\varsigma_{(\mu,k)} = \tilde\varsigma_{(\mu',k)} T^*_e$.

Now let $v = r(e)$, and observe that
\[
T^*_e \tilde\varsigma_{(v,k)} h_{\lambda ,x}
    = \begin{cases}
        \alpha_\lambda(\chi_{Z(v,k)})(x) h_{\lambda', x} &\text{ if $\lambda = e\lambda'$}\\
        0 &\text{ otherwise,}
    \end{cases}
\]
while
\[
\sum_{e\tau \in E^{n_k}} \tilde\varsigma_{(\tau,k)} T^*_e h_{\lambda ,x}
    = \begin{cases}
        \sum_{e\tau \in E^{n_k}} \alpha_{\lambda'}(\chi_{Z(\tau,k)})(x) h_{\lambda', x} &\text{ if $\lambda = e\lambda'$}\\
        0 &\text{ otherwise}.
    \end{cases}
\]
Again, if $\lambda \not= e\lambda'$, then both expressions are zero, so we suppose that
$\lambda = e\lambda'$. We have $\alpha_\lambda(\chi_{Z(v,k)})(x) = 1$ if and only if
$r(\lambda) = v$ and $|\lambda x_i| \in n_i\NN$ for large $i$. Also, $\sum_{e\tau \in
E^{n_k}} \alpha_{\lambda'}(\chi_{Z(\tau,k)})(x) = 1$ if and only if $[\lambda' x_i]_{n_i}
\in E^{n_i - 1}$ for large $i$, which is equivalent to $|\lambda' x_i| \equiv n_i - 1
\;(\operatorname{mod}\, n_i)$ for large $i$, and so $T^*_e \tilde\varsigma_{(v,k)}
h_{\lambda ,x} = \sum_{e\tau \in E^{n_k}} \tilde\varsigma_{(\tau,k)} T^*_e h_{\lambda
,x}$ as required.

Finally, suppose that $\mu \not= e\mu'$ and $\mu \not= r(e)$. We immediately see that
$T^*_e \tilde\varsigma_{(\mu,k)} = 0$ if $\mu \in E^0 \setminus r(e)$. If $\mu \not\in
E^0$, then $\mu_1 \not= e$, so that $\tilde\varsigma_{(\mu,k)}$ is the projection onto a
subspace of $\clsp\{h_{\lambda, x} : (\lambda x_i)_1 = \mu_1\text{ for large $i$}\}$,
which is orthogonal to the projection $T_e T^*_e$ onto $\clsp\{h_{\lambda, x} : \lambda_1
= e\}$.

We have now established that $(T, Q, \tilde\varsigma)$ is an $\omega$-representation, and
so the universal property of $\Tt(E, \omega)$ gives the desired homomorphism $\varsigma$.
\end{proof}

The following technical result will help in our construction of KMS states.

\begin{lem}\label{lem:m<->eps}
Let $E$ be a strongly connected finite directed graph with no sources, and take a
sequence $\omega = (n_k)^\infty_{k=1}$ of nonzero positive integers such that $n_k \mid
n_{k+1}$ for all $k$. Take $\beta > \ln\rho(A_E)$. The series $\sum^\infty_{j=0}
e^{-\beta j}A^j_\omega$ converges in norm to an inverse for $1 - e^{-\beta} A_\omega$.
For $\varepsilon \in \Mm^+(\varprojlim E^{<n_k})$ and $\tau \in E^{<n_k}$,
\[
    (1 - e^{-\beta}A_\omega)^{-1}(\varepsilon)(Z(\tau, k))
        = \sum_{(\lambda,\nu) \in \tau E(n_k)^*} e^{-\beta|\lambda|} \varepsilon(Z(\nu, k)).
\]
\end{lem}
\begin{proof}
Proposition~\ref{prp:Aomega eigenmeasure} gives $\|A_\omega\| = \rho(A_E)$. Since $\beta
> \ln\rho(A_E)$, we have $\|e^{-\beta}A_\omega\| < 1$, and so $\sum^\infty_{j=0} e^{-\beta
j}A^j_\omega$ converges in operator norm to $(1 - e^{-\beta}A_\omega)^{-1}$.

Now take $\tau \in E^{<n_k}$. Using Lemma~\ref{lem:Aomega on cylinders} at the second
equality, we calculate
\begin{align*}
(1 - e^{-\beta}A_\omega)^{-1}(\varepsilon)(Z(\tau, k))
    &= \sum^\infty_{j=0} e^{-\beta j}(A^j_\omega\varepsilon)(Z(\tau, k)) \\
    &= \sum^\infty_{j=0} \sum_{\nu \in E^{<n_k}} e^{-\beta j}|\tau E(n_k)^j \nu| \varepsilon(Z(\nu, k)) \\
    &= \sum^\infty_{j=0} \sum_{(\lambda,\nu) \in \tau E(n_k)^j} e^{-\beta j} \varepsilon(Z(\nu, k)) \\
    &= \sum_{(\lambda,\nu) \in \tau E(n_k)^*} e^{-\beta |\lambda|} \varepsilon(Z(\nu, k)).\qedhere
\end{align*}
\end{proof}

We can now construct a KMS state for each measure that satisfies the subinvariance
relation in Theorem~\ref{thm:KMSchar}(\ref{it:KMS subinvariance}).

\begin{prp}\label{prp:constructKMS}
Let $E$ be a strongly connected finite directed graph with no sources, and take a
sequence $\omega = (n_k)^\infty_{k=1}$ of nonzero positive integers such that $n_k \mid
n_{k+1}$ for all $k$. Take $\beta > \ln\rho(A_E)$. Suppose that $m \in
\Mm^+_1(\varprojlim E^{<n_k})$ satisfies $A_\omega m \le e^\beta m$. Then there is a
KMS$_\beta$ state $\phi_m$ of $(\Tt(E, \omega), \alpha)$ satisfying
\begin{equation}\label{eq:phim formula}
\phi_m(t_\mu \pi_{(\tau, k)} t^*_\nu) = \delta_{\mu,\nu} e^{-\beta|\mu|} m(Z(\tau, k))
\end{equation}
for all $\tau \in E^{<n_k}$ and all $\mu,\nu \in E^* r(\tau)$.
\end{prp}
\begin{proof}
Let $\varepsilon := (1 - e^{-\beta}A_\omega) m$. Since $m$ is subinvariant, $\varepsilon$
is a positive measure on $\varprojlim E^{<n_k}$. Let $\varsigma : \Tt(E, \omega) \to
\Bb(\ell^2(E^* \times_{E^0} \varprojlim E^{<n_k}))$ be the representation of
Proposition~\ref{prp:path rep}. We aim to define $\phi_m$ by
\begin{equation}\label{eq:phim def}
\phi_m(a) = \sum_{\lambda \in E^*} e^{-\beta |\lambda|} \int_{x \in \varprojlim E^{<n_k}}
    \chi_{Z(s(\lambda), 1)}(x) \big(\varsigma(a)h_{\lambda, x} \mid h_{\lambda, x}\big)\,d\varepsilon(x).
\end{equation}
We first show that for $a \in \Tt(E, \omega)$, the function $f_a : E^* \times_{E^0}
\varprojlim E^{<n_k} \to \CC$ given by $f_a(\lambda, x) = (\varsigma(a)h_{\lambda, x}
\mid h_{\lambda, x})$ is integrable. First consider $a = t_\mu \pi_{(\tau, k)} t^*_\nu$.
We have
\begin{align}
    \big(\varsigma(t_\mu \pi_{(\tau, k)} t^*_\nu) h_{\lambda, x} \mid h_{\lambda, x}\big)
        &= \big(\varsigma(\pi_{(\tau, k)} t^*_\nu) h_{\lambda, x} \mid \varsigma(\pi_{(\tau, k)} t^*_\mu) h_{\lambda, x}\big)\nonumber\\
        &= \begin{cases}
            \alpha_{\lambda'}(\chi_{Z(\tau, k)})(x) &\text{ if $\lambda = \nu\lambda' = \mu\lambda'$} \\
            0 &\text{ otherwise.}
        \end{cases}\label{ep:i.p.}
\end{align}
So $f_a$ is the characteristic function of the clopen set $\bigsqcup\{Z(\tau, k) : \tau
\in E^{<n_k}, [\lambda\tau]_{n_k} = \mu\}$, and hence integrable. Consequently $f_a$ is
integrable for $a \in \lsp\{t_\mu \pi_{(\tau, k)} t^*_\nu\}$. Now as in
\cite[Lemma~10.1(b)]{anHuefLacaEtAl:JFA15}, for $a \in \Tt(E, \omega)$ is a pointwise
limit of integrable functions and hence itself integrable as claimed.

Since each $Z(s(\lambda), 1)$ is measurable, the functions $\chi_{Z(s(\lambda), 1)}
f_a$ are also integrable. Since $f_a(\lambda, x) \le \|a\|$ for all $(\lambda, x)$, we
have $\int_{\varprojlim E^{<n_k}} \chi_{Z(s(\lambda), 1)} f_a(\lambda, x)\,d\mu(x) <
\|a\|$. Since $\beta > \ln\rho(A_E)$, Lemma~\ref{lem:m<->eps} implies that $\sum_{\lambda
\in E^*v} e^{-\beta|\lambda|}$ is convergent for each $v$, and so the series on the
right-hand side of~\eqref{eq:phim def} is bounded above by the convergent series $\sum_{v
\in E^0} \sum_{\lambda \in E^*v} e^{-\beta|\lambda|} \|a\|$, and hence itself convergent.
So there is a bounded linear map $\phi_m : \Tt(E, \omega) \to \CC$
satisfying~\eqref{eq:phim def}.

This $\phi_m$ is positive because $f_{a^*a}$ is positive-valued. We check that $\phi_m$
is a state. We use Lemma~\ref{lem:m<->eps} at the penultimate equality to calculate
\begin{align*}
\phi_m(1)
    &= \sum_{\lambda \in E^*} e^{-\beta |\lambda|} \int_{x \in \varprojlim E^{<n_k}}
    \chi_{Z(s(\lambda), 1)}(x)\,d\varepsilon(x) \\
    &= \sum_{\lambda \in E^*} e^{-\beta|\lambda|} \varepsilon(Z(s(\lambda), 1))
    = \sum_{w \in E^0} m(Z(w,1)) = 1.
\end{align*}
Since $\mu\lambda' = \nu\lambda'$ forces $\mu = \nu$, we have $\phi_m(t_\mu
\pi_{(\tau,k)} t^*_\nu) = 0$ if $\mu \not= \nu$. Moreover, each
\[
\big(\varsigma(t_\mu \pi_{(\tau, k)} t^*_\mu) h_{\lambda, x} \mid h_{\lambda, x}\big)
    = \|\varsigma(\pi_{(\tau, k)} t^*_\mu) h_{\lambda, x}\|^2
    = \begin{cases}
        \alpha_{\lambda'}(\chi_{Z(\tau,k)})(x) &\text{ if $\lambda = \mu\lambda'$}\\
        0 &\text{ otherwise.}
    \end{cases}
\]
Hence
\begin{align*}
\phi_m(t_\mu \pi_{(\tau, k)} t^*_\mu)
    &= \sum_{\mu\lambda' \in E^*} e^{-\beta |\mu\lambda'|} \int_{x \in \varprojlim E^{<n_k}}
    \alpha_{\lambda'}(\chi_{Z(\tau,k)})(x)\,d\varepsilon(x)\\
    &= e^{\beta|\mu|} \sum_{\lambda' \in s(\mu)E^*} e^{-\beta|\lambda'|}
        \int_{x \in Z(s(\lambda'), 1)} \chi_{Z(\tau, k)}\big((r_{n_i}(\lambda', x_i))^\infty_{i=1}\big)\,d\varepsilon(x) \\
    &= e^{\beta|\mu|} \sum_{\lambda' \in s(\mu)E^*} e^{-\beta|\lambda'|}
        \varepsilon\big(\{x : r_{n_k}(\lambda', x_k) = \tau\}\big) \\
    &= e^{\beta|\mu|} \sum_{(\lambda',\nu) \in \tau E(n_k)^*} e^{-\beta|\lambda'|} \varepsilon(Z(\nu,k))\\
    &= e^{\beta|\mu|} m(Z(\tau, k)),
\end{align*}
which is~\eqref{eq:phim formula}. Putting $\mu = r(\tau)$ gives $\phi_m(\pi_{(\tau, k)})
= m(Z(\tau, k))$, and so $\phi_m$ also satisfies~\eqref{eq:KMS on spanner}, and is
therefore KMS by Theorem~\ref{thm:KMSchar}(\ref{it:KMS cond}).
\end{proof}

\begin{thm}\label{thm:affine iso}
Let $E$ be a strongly connected finite directed graph with no sources, and take a
sequence $\omega = (n_k)^\infty_{k=1}$ of nonzero positive integers such that $n_k \mid
n_{k+1}$ for all $k$. Let $\alpha : \RR \to \Aut(\Tt(E, \omega))$ be given by $\alpha_t =
\gamma_{e^{it}}$. Take $\beta > \ln\rho(A_E)$.
\begin{enumerate}
\item\label{it:which epsilons} Take $\varepsilon \in \Mm^+(\varprojlim E^{<n_k})$.
    For each $x \in \varprojlim E^{<n_k}$, the series $\sum_{\mu \in E^* r(x)}
    e^{-\beta|\mu|}$ converges; we write $y(x)$ for its limit. We have $(1 -
    e^{-\beta} A_\omega)^{-1}\varepsilon \in \Mm^+_1(\varprojlim E^{<n_k})$ if and
    only if
    \[
        \int_{x \in \varprojlim E^{<n_k}} y(x)\,d\varepsilon(x) = 1.
    \]
\item\label{it:eps->state} Suppose that $\varepsilon \in \Mm^+(\varprojlim E^{<n_k})$
    satisfies $\int_{\varprojlim E^{<n_k}} y(x)\,d\varepsilon(x) = 1$, and define $m
    := (1 - e^{-\beta} A_\omega)^{-1}\varepsilon$. There is a KMS$_\beta$ state
    $\phi_\varepsilon$ of $(\Tt(E, \omega), \alpha)$ such that
    \begin{equation}\label{eq:phieps on spanner}
        \phi_\varepsilon(t_\mu \pi_{(\tau,k)} t^*_\nu)
            = \delta_{\mu,\nu} e^{-\beta|\mu|} m(Z(\tau, k)).
    \end{equation}
\item\label{it:affineiso} The map $\varepsilon \mapsto \phi_\varepsilon$ is an affine
    isomorphism of
    \[\textstyle
        \Omega_\beta := \{\varepsilon \in \Mm^+(\varprojlim E^{<n_k}) : \int y(x)\,d\varepsilon(x) = 1\}
    \]
    onto the simplex of KMS$_\beta$ states of $(\Tt(E, \omega), \alpha)$. The inverse
    of this isomorphism takes a KMS$_\beta$ state $\phi$ to $(1 -
    e^{-\beta}A_\omega)m^\phi$.
\end{enumerate}
\end{thm}
\begin{proof}
(\ref{it:which epsilons}) The series $\sum^\infty_{j=0} (e^{-\beta
j}A^j_\omega)\varepsilon$ converges to $m := (1 - e^{-\beta} A_\omega)^{-1}\varepsilon$
because $\beta > \ln\rho(A_E)$. This shows that $m \ge 0$.

Using Lemma~\ref{lem:m<->eps}, we fix $k$ and calculate
\begin{align*}
m(\varprojlim E^{<n_k})
    &= \sum_{(\lambda,\nu) \in E(n_k)^*} e^{-\beta|\lambda|} \varepsilon(Z(\nu, k))
     = \sum_{\nu \in E^{<n_k}} \sum_{\lambda \in E^* r(\nu)} e^{-\beta|\lambda|} \varepsilon(Z(\nu, k)) \\
    &= \sum_{\nu \in E^{<n_k}} \int_{x \in Z(\nu, k)} y(x)\,d\varepsilon(x)
     = \int_{x \in \varprojlim E^{<n_k}} y(x)\,d\varepsilon(x).
\end{align*}

(\ref{it:eps->state}) We claim that $A_\omega m \le e^\beta m$. We calculate
\[
A_\omega m
    = A_\omega\Big(\sum^\infty_{j=0} e^{-\beta j} A^j_\omega\Big)\varepsilon
    = e^\beta \Big(\sum^\infty_{j=1} e^{-\beta j} A^j_\omega\Big)\varepsilon
    \le e^\beta \Big(\sum^\infty_{j=0} e^{-\beta j} A^j_\omega\Big)\varepsilon
    = e^\beta m.
\]
Now Proposition~\ref{prp:constructKMS} gives a KMS$_\beta$ state $\phi_\varepsilon$ satisfying~\eqref{eq:phieps on spanner}.

(\ref{it:affineiso}) We claim that every KMS$_\beta$ state $\phi$ has the form
$\phi_\varepsilon$. Fix a KMS$_\beta$ state $\phi$, and let $m^\phi$ be the measure such
that $m^\phi(Z(\mu,k)) = \phi(\pi_{(\mu,k)})$. By part~(\ref{it:KMS subinvariance}),
$m^\phi$ is a subinvariant probability measure. Let $\varepsilon := (1 -
e^{-\beta}A_\omega)^{-1} m^\phi$. Then $m^\phi = (1 - e^{-\beta}A_\omega) \varepsilon$ by
construction, and comparing~\eqref{eq:phieps on spanner} with~\eqref{eq:KMS on spanner}
shows that $\phi = \phi_\varepsilon$.

The formula~\eqref{eq:phieps on spanner} also shows that the map $F : \varepsilon \to
\phi_\varepsilon$ is injective and weak$^*$-continuous from $\Omega_\beta$ to the state
space of $\Tt(E, \omega)$. We have just seen that it is surjective onto the KMS$_\beta$
simplex, which is compact since $C^*(E, \omega)$ is unital. Hence $F$ is a homeomorphism
of $\Omega_\beta$ onto the KMS$_\beta$ simplex. The formula~\eqref{eq:phim def} shows
that $F$ is affine, and the formula for the inverse follows from our proof of
surjectivity in the preceding paragraph.
\end{proof}

\begin{cor}\label{cor:normalisedaffine}
Let $E$ be a strongly connected finite directed graph with no sources, and take a
sequence $\omega = (n_k)^\infty_{k=1}$ of nonzero positive integers such that $n_k \mid
n_{k+1}$ for all $k$. Let $\alpha : \RR \to \Aut(\Tt(E, \omega))$ be given by $\alpha_t =
\gamma_{e^{it}}$. Take $\beta > \ln\rho(A_E)$. Let $y$ be as in part~(\ref{it:which
epsilons}) of Theorem~\ref{thm:affine iso}. The map $m \mapsto \phi_{y^{-1}m}$ is an
affine isomorphism of $\Mm^+_1(\varprojlim E^{<n_k})$ onto the KMS$_\beta$-simplex of
$(\Tt(E, \omega), \alpha)$.
\end{cor}
\begin{proof}
Since $y$ takes strictly positive values and is bounded, the map $m \mapsto y^{-1}m$ is
an affine isomorphism of $\Mm^+_1(\varprojlim E^{<n_k})$ onto $\Omega_\beta$, so the
result follows from Theorem~\ref{thm:affine iso}(\ref{it:affineiso}).
\end{proof}

\subsection{KMS states at the critical temperature}
We show that the extreme KMS states at the critical temperature $\ln\rho(A_E)$ are
indexed by the equivalence classes $E^0/{\sim_{n_k}}$ of Lemma~\ref{lem:Cn} for any $k$
such that $\gcd(\Pp_E, n_k) = \gcd(\Pp_E, \omega)$.

\begin{thm}\label{thm:betterKMS}
Let $E$ be a strongly connected finite directed graph with no sources, and take a
sequence $\omega = (n_k)^\infty_{k=1}$ of nonzero positive integers such that $n_k \mid
n_{k+1}$ for all $k$. Fix $k$ such that $\gcd(\Pp_E, n_k) = \gcd(\Pp_E, \omega)$, and let
$\sim_{n_k}$ be the equivalence relation on $E^0$ of Lemma~\ref{lem:Cn}. Let $\alpha :
\RR \to \Aut(\Tt(E, \omega))$ be given by $\alpha_t = \gamma_{e^{it}}$. Let $x^E$ be the
unimodular Perron--Frobenius eigenvector of $A_E$.
\begin{enumerate}
\item\label{it:criticalKMS} For each $\Lambda \in E^0/{\sim_{n_k}}$, there is a
    KMS$_{\ln\rho(A_E)}$ state $\phi^\Lambda$ for $(\Tt(E, \omega), \alpha)$
    satisfying
    \begin{equation}\label{eq:criticalKMS on spanners}
        \phi^\Lambda(t_\mu \pi_{(\tau, k)} t^*_\nu) = \chi_\Lambda(s(\tau)) \delta_{\mu,\nu} \frac{1}{\sum_{v \in \Lambda} x^E_v} \rho(A_E)^{-|\mu|-|\tau|} x^E_{s(\tau)}.
    \end{equation}
    This is the unique KMS$_{\ln\rho(A_E)}$ state for $(\Tt(E, \omega), \alpha)$
    satisfying $\phi^\Lambda(\pi_{(v,k)}) = 0$ for all $v \in E^0 \setminus \Lambda$,
    and it factors through a KMS$_{\ln\rho(A_E)}$ state $\overline{\phi}^\Lambda$ of
    $(C^*(E, \omega), \alpha)$.
\item\label{it:uniqueKMS} The states $\overline{\phi}^\Lambda$ are the extreme points
    of the KMS$_{\ln\rho(A_E)}$-simplex of $(C^*(E, \omega), \alpha)$, and there are
    no KMS$_\beta$-states for $(C^*(E, \omega), \alpha)$ for any $\beta \not=
    \ln\rho(A_E)$.
\end{enumerate}
\end{thm}
\begin{proof}
(\ref{it:criticalKMS}) Fix $\Lambda \in E^0/{\sim_{n_k}}$. We first prove the existence
of a KMS$_{\ln\rho(A_E)}$ state satisfying~\eqref{eq:criticalKMS on spanners}. For each
$l \ge k$, let $E(n_l)_\Lambda$ be the component of $E(n_l)$ with vertices
$E^{<n_l}\Lambda$. Theorem~4.3(a) of \cite{anHuefLacaEtAl:ETDS14} shows that there is a
unique KMS state $\phi^\Lambda_l$ of $C^*(E(n_l)) = C^*(E, n_l)$ that vanishes on
$\varepsilon_\mu$ for $\mu \in E^{<n_l}(E^0 \setminus \Lambda)$. Since each
$\phi^\Lambda_{l+1}$ must restrict to a KMS state of $C^*(E, n_l)$, the $\phi^\Lambda_l$
are compatible under the inclusions $C^*(E, n_l) \hookrightarrow C^*(E, n_{l+1})$. So
continuity yields a state $\phi^\Lambda$ on $C^*(E, \omega)$ that agrees with each
$\phi^\Lambda_{n_l}$ on the image of $C^*(E, n_l)$, and hence
satisfies~\eqref{eq:criticalKMS on spanners}. It follows that $\phi^\Lambda(\pi_{(v,k)})
= 0$ for all $v \in E^0 \setminus \Lambda$. Uniqueness follows from uniqueness of the
$\phi^\Lambda_l$. Theorem~\ref{thm:KMSchar}(\ref{it:KMS factor}) shows that
$\phi^\Lambda$ factors through $(C^*(E, \omega), \alpha)$.

(\ref{it:uniqueKMS}) The $\phi^\Lambda$ are linearly independent, and so
are the extreme points of the convex set they generate. So it suffices to show that every
KMS state of $C^*(E, \omega)$ is a convex combination of the $\phi^\Lambda$. Suppose that
$\psi$ is a KMS$_\beta$ state of $(C^*(E, \omega), \alpha)$. Let $q : \Tt(E, \omega) \to
C^*(E, \omega)$ be the quotient map. Theorem~\ref{thm:KMSchar}(\ref{it:KMS factor})
implies that $A_\omega m^{\psi\circ q} = e^\beta m^{\psi \circ q}$. Hence
Lemma~\ref{lem:uniquePF}(\ref{it:unique}) shows that $m^{\psi \circ q}$ is a convex
combination $m^{\psi \circ q} = \sum_\Lambda t_\Lambda m^\Lambda$ of the $m^\Lambda$. It
then follows from Theorem~\ref{thm:KMSchar}(\ref{it:KMS factor}) that $\psi \circ q =
\sum_\Lambda t_\Lambda \overline{\phi}^\Lambda$. Theorem~\ref{thm:KMSchar}(\ref{it:KMS
factor}) combined with Lemma~\ref{lem:uniquePF}(\ref{it:unique}) shows that there are no
KMS states for $C^*(E, \omega)$ at any other inverse temperature.
\end{proof}

\begin{proof}[Proof of Theorem~\ref{thm:mainKMS}]
Item~(\ref{it:main1}) follows from Corollary~\ref{cor:normalisedaffine} and
item~(\ref{it:main2}) follows from Theorem~\ref{thm:betterKMS}.

For item~(\ref{it:main4}), recall that Theorem~\ref{thm:KMSchar}(\ref{it:KMS factor})
implies that a KMS$_\beta$ state $\phi$ factors through $C^*(E, \omega)$ if and only if
$A_\omega m^\phi = e^{-\beta} m^\phi$. If $\phi$ factors through $C^*(E, \omega)$, then
$m^\phi$ is a positive eigenmeasure for $A_\omega$ and Lemma~\ref{lem:uniquePF} gives
$\beta = \ln\rho(A_E)$. On the other hand, if $\beta = \ln\rho(A_E)$, then
Theorem~\ref{thm:KMSchar}(\ref{it:KMS subinvariance}) gives $A_\omega m^\phi \le
\rho(A_E) m^\phi$, and then Lemma~\ref{lem:subinvariance} forces equality.

Finally, for~(\ref{it:main3}), suppose that $\phi$ is a KMS$_\beta$ state of $(\Tt(E,
\omega), \alpha)$. Then Theorem~\ref{thm:KMSchar}(\ref{it:KMS subinvariance}) implies
that $A_\omega m^\phi \le e^{\beta} m^\phi$, and then Lemma~\ref{lem:subinvariance} gives
$e^\beta \ge \rho(A_E)$ and hence $\beta \ge \ln\rho(A_E)$.
\end{proof}

We deduce that simplicity of $C^*(E, \omega)$ is reflected by the existence of a unique
KMS state for the gauge action.

\begin{prp}\label{prp:simple<->KMS}
Let $E$ be a strongly connected finite directed graph with no sources, and take a
sequence $\omega = (n_k)^\infty_{k=1}$ of nonzero positive integers such that $n_k \mid
n_{k+1}$ for all $k$ and $n_k \to \infty$. Let
$\alpha : \RR \to \Aut(\Tt(E, \omega))$ be given by $\alpha_t = \gamma_{e^{it}}$. The
following are equivalent
\begin{enumerate}
\item\label{it:gcd1} $\gcd(\Pp_E, \omega) = 1$;
\item\label{it:simple} $C^*(E,\omega)$ is simple;
\item\label{it:phifactors} there is a unique KMS state for $(C^*(E,\omega), \alpha)$ and
the state ~\eqref{eq:criticalKMS on spanners} factors through this state; and
\item\label{it:factorstate} the state~\eqref{eq:criticalKMS on spanners} is a factor
    state.
\end{enumerate}
\end{prp}
\begin{proof}
Corollary~\ref{cor:simple} gives \mbox{(\ref{it:gcd1})${}\iff{}$(\ref{it:simple})}, and
Theorem~\ref{thm:betterKMS} gives
\mbox{(\ref{it:gcd1})${}\implies{}$(\ref{it:phifactors})}. To establish
\mbox{(\ref{it:phifactors})${}\implies{}$(\ref{it:factorstate})}, suppose that $\phi$
factors through the unique KMS state of $(C^*(E,\omega), \alpha)$. Then it is an extreme
point of the KMS simplex and hence a factor state by
\cite[Theorem~5.3.30(3)]{BratteliRobinson:OAQSMvII}.

For \mbox{(\ref{it:factorstate})${}\implies{}$(\ref{it:gcd1})} let $\phi$ be the state
given by~\eqref{eq:criticalKMS on spanners} and suppose that $\phi$ is a factor state for
$\Tt C^*(E,\omega)$. Fix $k$ such that $\gcd(\Pp_E,n_k) = \gcd(\Pp,\omega)$. Recall
the equivalence relation $\sim_{n_k}$ of Lemma~\ref{lem:Cn} and the projections
$Q_{k,\Lambda}$ of Lemma~\ref{lem:direct sum}. We have $\phi(\pi_{(\mu,k)}) =
\frac{1}{n_k}\rho(A_E)^{-|\mu|}
x^E_{s(\mu)} \not= 0$ for all $\mu$ because the Perron--Frobenius eigenvector has
strictly positive entries. So each $\phi(Q_{k, \Lambda}) \not= 0$. So the GNS
representation $\pi_\phi$ is also nonzero on the $Q_{k,\Lambda}$. Lemma~\ref{lem:direct
sum} implies that the $Q_{k,\Lambda}$ are central in $\Tt(E, \omega)$, and so the
$\pi_\phi(Q_{k,\Lambda})$ are mutually orthogonal central projections in
$\pi_\phi(\Tt(E,\omega))''$. Since $\phi$ is a factor state, it follows that there is
only one equivalence class $\Lambda$ for $\sim_{n_k}$, and so $\gcd(\Pp_E, \omega) = 1$.
\end{proof}

\end{document}